\documentclass[11pt]{article}\UseRawInputEncoding
\usepackage{amssymb,amsfonts,amsmath,latexsym,epsf,tikz,url}
\usepackage[usenames,dvipsnames]{pstricks}
\usepackage{pstricks-add}
\usepackage{epsfig}
\usepackage{pst-grad} 
\usepackage{pst-plot} 
\usepackage[space]{grffile} 
\usepackage{etoolbox} 
\makeatletter 
\patchcmd\Gread@eps{\@inputcheck#1 }{\@inputcheck"#1"\relax}{}{}
\makeatother

\newtheorem{theorem}{Theorem}[section]

\newcommand{\qed}{\hfill $\square$\medskip}

\begin{document}

\title{Mostar index and edge Mostar index of polymers}

\author{
Nima Ghanbari
\and
Saeid Alikhani  $^{}$\footnote{Corresponding author}
}

\date{\today}

\maketitle

\begin{center}
Department of Informatics, University of Bergen, P.O. Box 7803, 5020 Bergen, Norway\\
Department of Mathematics, Yazd University, 89195-741, Yazd, Iran\\
{\tt   Nima.ghanbari@uib.no,alikhani@yazd.ac.ir }
\end{center}


\begin{abstract}
Let $G=(V,E)$ be a graph and $e=uv\in E$. Define $n_u(e,G)$  be the number of vertices of $G$ closer to $u$ than to $v$.  The number $n_v(e,G)$ can be defined in an analogous way. The Mostar index of $G$ is a new graph invariant  defined as $Mo(G)=\sum_{uv\in E(G)}|n_u(uv,G)-n_v(uv,G)|$. The edge version of Mostar index is defined as $Mo_e(G)=\sum_{e=uv\in E(G)} |m_u(e|G)-m_v(G|e)|$, where $m_u(e|G)$ and $m_v(e|G)$ are the number of edges of $G$ lying closer to vertex $u$ than to vertex $v$ and the number of edges of $G$ lying closer to vertex $v$ than to vertex $u$, respectively.  Let $G$ be a connected graph constructed from pairwise disjoint connected graphs $G_1,\ldots ,G_k$ by selecting a vertex of $G_1$, a vertex of $G_2$, and identifying these two
 vertices. Then continue in this manner inductively. We say that $G$ is a polymer graph, obtained by point-attaching   from monomer units $G_1,...,G_k$.  
 In this paper, we consider some  particular cases  of these graphs that  are  of importance in chemistry  and study their Mostar and edge Mostar indices. 
\end{abstract}


\medskip
\noindent{\bf AMS Subj.\ Class.:} 05C09, 05C92.

\section{Introduction}

 A molecular graph is a simple graph such that its vertices correspond to the atoms and the edges to the bonds of a molecule. 
 Let $G = (V, E)$ be a finite, connected, simple graph.  
 A topological index of $G$ is a real number related to $G$. It does not depend on the labeling or pictorial representation of a graph. The Wiener index $W(G)$ is the first distance based topological index defined as $W(G) = \sum_{\{u,v\}\subseteq G}d(u,v)=\frac{1}{2} \sum_{u,v\in V(G)} d(u,v)$ with the summation runs over all pairs of vertices of $G$ \cite{20}.
 The topological indices and graph invariants based on distances between vertices of a graph are widely used for characterizing molecular graphs, establishing relationships between structure and properties of molecules, predicting biological activity of chemical compounds, and making their chemical applications.  The
 Wiener index is one of the most used topological indices with high correlation with many physical and chemical indices of molecular compounds \cite{20}. 
 In a recent paper, Dosli\'c et al. \cite{Doslic} introduced a newbond-additive structural invariant as a quantitative refinement of the distance nonbalancedness and also a measure of peripherality  in  graphs.  They  used  the  name  Mostar index  for  this  invariant  which  is  defined  as $Mo(G)=\sum_{uv\in E(G)}|n_u(uv,G)-n_v(uv,G)|$,   where $n_u(uv,G)$ is  the number of vertices of $G$ closer to $u$ than to $v$, and similarly, $n_v(uv,G)$ is the number of vertices closer to $v$ than to $u$.  They determined  the  extremal  values  of  this  invariant  and characterized extremal trees and unicyclic graphs with respect to the Mostar index. S. Akhter in \cite{Akhter} computed the Mostar index of corona product, Cartesian product, join, lexicographic product, Indu-Bala product and subdivisionvertex-edge join of graphs and applied  results to find the Mostar index of various classes of chemical graphs and nanostructures. The Mostar index of bicyclic
 graphs was studied by Tepeh \cite{15}. A cacti graph is a graph in which any block is either a cut edge or a cycle, or equivalently, a graph in which
 any two cycles have at most one common vertex. Hayat and Zhou in \cite{Filomat} gave an upper bound for the Mostar
 index of cacti of order $n$ with $k$ cycles, and also they  characterized those cacti that achieve the bound. 
 
 The edge version of Mostar index has considered in \cite{edge, Liu} and is defined as
 $Mo_e(G)=\sum_{e=uv\in E(G)} |m_u(e|G)-m_v(G|e)|$, where $m_u(e|G)$ and $m_v(e|G)$ are the number of edges of $G$ lying closer to vertex $u$ than to vertex $v$ and the number of edges of $G$ lying closer to vertex $v$ than to vertex $u$, respectively.  Liu et.al, in \cite{Liu} determined the extremal values of edge Mostar index of some graphs such as trees and unicyclic graphs.

 \medskip
 
 In this paper, we consider the Mostar  index and the edge Mostar index of polymer graphs. Such graphs can be decomposed into subgraphs that we call monomer units. Blocks of graphs are particular examples of monomer units, but a monomer unit may consist of several blocks. For convenience, the definition of these kind of graphs will be given in the next  section.  In Section 2,  the Mostar  index of some graphs are computed  from their monomer units. In Section 3, we obtain the Mostar  index and the edge Mostar index  of
 families of graphs that are of importance in chemistry.

\section{Mostar index and edge Mostar index of polymers}

\begin{figure}
	\begin{center}
		\psscalebox{0.6 0.6}
		{
			\begin{pspicture}(0,-4.819607)(13.664668,2.90118)
			\pscircle[linecolor=black, linewidth=0.04, dimen=outer](5.0985146,1.0603933){1.6}
			\pscustom[linecolor=black, linewidth=0.04]
			{
				\newpath
				\moveto(11.898515,0.66039336)
			}
			\pscustom[linecolor=black, linewidth=0.04]
			{
				\newpath
				\moveto(11.898515,0.26039338)
			}
			\pscustom[linecolor=black, linewidth=0.04]
			{
				\newpath
				\moveto(12.698514,0.66039336)
			}
			\pscustom[linecolor=black, linewidth=0.04]
			{
				\newpath
				\moveto(10.298514,1.0603933)
			}
			\pscustom[linecolor=black, linewidth=0.04]
			{
				\newpath
				\moveto(11.098515,-0.9396066)
			}
			\pscustom[linecolor=black, linewidth=0.04]
			{
				\newpath
				\moveto(11.098515,-0.9396066)
			}
			\pscustom[linecolor=black, linewidth=0.04]
			{
				\newpath
				\moveto(11.898515,0.66039336)
			}
			\pscustom[linecolor=black, linewidth=0.04]
			{
				\newpath
				\moveto(11.898515,-0.9396066)
			}
			\pscustom[linecolor=black, linewidth=0.04]
			{
				\newpath
				\moveto(11.898515,-0.9396066)
			}
			\pscustom[linecolor=black, linewidth=0.04]
			{
				\newpath
				\moveto(12.698514,-0.9396066)
			}
			\pscustom[linecolor=black, linewidth=0.04]
			{
				\newpath
				\moveto(12.698514,0.26039338)
			}
			\pscustom[linecolor=black, linewidth=0.04]
			{
				\newpath
				\moveto(14.298514,0.66039336)
				\closepath}
			\psbezier[linecolor=black, linewidth=0.04](11.598515,1.0203934)(12.220886,1.467607)(12.593457,1.262929)(13.268515,1.0203933715820312)(13.943572,0.7778577)(12.308265,0.90039337)(12.224765,0.10039337)(12.141264,-0.69960666)(10.976142,0.5731798)(11.598515,1.0203934)
			\psbezier[linecolor=black, linewidth=0.04](4.8362556,-3.2521083)(4.063277,-2.2959895)(4.6714916,-1.9655427)(4.891483,-0.99004078729821)(5.111474,-0.014538889)(5.3979383,-0.84551746)(5.373531,-1.8452196)(5.349124,-2.8449216)(5.6092343,-4.208227)(4.8362556,-3.2521083)
			\psbezier[linecolor=black, linewidth=0.04](8.198514,-2.0396066)(6.8114076,-1.3924998)(6.844908,-0.93520766)(5.8785143,-1.6996066284179687)(4.9121203,-2.4640057)(5.6385145,-3.4996066)(6.3385143,-2.8396065)(7.0385146,-2.1796067)(9.585621,-2.6867135)(8.198514,-2.0396066)
			\pscircle[linecolor=black, linewidth=0.04, dimen=outer](7.5785146,-3.6396067){1.18}
			\psdots[linecolor=black, dotsize=0.2](11.418514,0.7403934)
			\psdots[linecolor=black, dotsize=0.2](9.618514,1.5003934)
			\psdots[linecolor=black, dotsize=0.2](6.6585145,0.7403934)
			\psdots[linecolor=black, dotsize=0.2](3.5185144,0.96039337)
			\psdots[linecolor=black, dotsize=0.2](5.1185145,-0.51960665)
			\psdots[linecolor=black, dotsize=0.2](5.3985143,-2.5796065)
			\psdots[linecolor=black, dotsize=0.2](7.458514,-2.4596066)
			\rput[bl](8.878514,0.42039338){$G_i$}
			\rput[bl](7.478514,-4.1196065){$G_j$}
			\psbezier[linecolor=black, linewidth=0.04](0.1985144,0.22039337)(0.93261385,0.89943534)(2.1385605,0.6900083)(3.0785143,0.9403933715820313)(4.0184684,1.1907784)(3.248657,0.442929)(2.2785144,0.20039338)(1.3083719,-0.042142253)(-0.53558505,-0.45864862)(0.1985144,0.22039337)
			\psbezier[linecolor=black, linewidth=0.04](2.885918,1.4892112)(1.7389486,2.4304078)(-0.48852357,3.5744174)(0.5524718,2.1502930326916756)(1.5934672,0.7261687)(1.5427756,1.2830372)(2.5062277,1.2429687)(3.46968,1.2029002)(4.0328875,0.5480146)(2.885918,1.4892112)
			\psellipse[linecolor=black, linewidth=0.04, dimen=outer](9.038514,0.7403934)(2.4,0.8)
			\psbezier[linecolor=black, linewidth=0.04](9.399693,1.883719)(9.770389,2.812473)(12.016343,2.7533927)(13.011008,2.856550531577144)(14.005673,2.9597082)(13.727474,2.4925284)(12.761896,2.2324166)(11.796317,1.9723049)(9.028996,0.9549648)(9.399693,1.883719)
			\end{pspicture}
		}
	\end{center}
	\caption{\label{Figure1} A polymer graph with monomer units  $G_1,\ldots , G_k$.}
\end{figure}
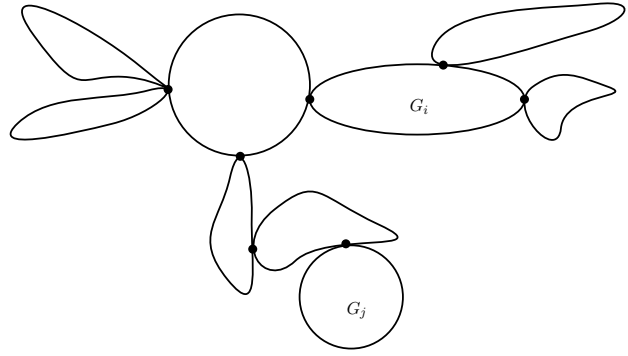

Let $G$ be a connected graph constructed from pairwise disjoint connected graphs
$G_1,\ldots ,G_k$ as follows. Select a vertex of $G_1$, a vertex of $G_2$, and identify these two vertices. Then continue in this manner inductively.  Note that the graph $G$ constructed in this way has a tree-like structure, the $G_i$'s being its building stones (see Figure \ref{Figure1}).  Usually  say that $G$ is a polymer graph, obtained by point-attaching from $G_1,\ldots , G_k$ and that $G_i$'s are the monomer units of $G$. A particular case of this construction is the decomposition of a connected graph into blocks (see \cite{Sombor,Deutsch}).

The following theorem is easy result which obtain by the definition of Mostar index, edge Mostar index  and point-attaching graph.

\begin{theorem}
	If  $G$ is a polymer graph with the monomer units  $G_1,\ldots , G_k$, then  $Mo(G)> \sum_{i=1}^{n}Mo(G_i),$ and $Mo_e(G)> \sum_{i=1}^{n}Mo_e(G_i)$. 
\end{theorem}

We consider some  particular cases  of point-attaching  graphs  and study their Mostar   and edge Mostar index . 
As an example of point-attaching graph,   consider the graph $K_m$ and $m$ copies of  $K_n$. By definition, the graph $Q(m, n)$ is obtained by identifying each vertex of $K_m$ with a vertex of a unique $K_n$. The graph $Q(5,4)$ is shown in Figure \ref{qmn}.

	\begin{theorem}
		For the graph $Q(m,n)$ (see Figure \ref{qmn}), we have:
	\begin{enumerate} 
		\item[(i)]
		$Mo(Q(m,n))=mn(m-1)(n-1).$
		
		\item[(ii)] 
		$Mo_e(Q(m,n))=\frac{m(n-1)(m-1)}{2}(n^2-n+m).$
		\end{enumerate}
	\end{theorem}

	\begin{proof}
		\begin{enumerate} 
			\item[(i)]
		First consider the edge $u_iu_j$ in $K_m$.  There are $n-1$ vertices which are closer to $u_i$ than $u_j$, and there are  $n-1$ vertices closer to $u_j$ than $u_i$. So $|n_{u_i}(u_iu_j,Q(m,n))-n_{u_j}(u_ju_i,Q(m,n))|=0$. Now consider the edge $vw$ in the $i$-th $K_n$. There is no vertices which are closer to $v$ than $w$, and visa versa. So $|n_{v}(vw,Q(m,n))-n_{w}(vw,Q(m,n))|=0$. Finally, consider the edge $u_iv$ in the $i$-th $K_n$. There are $n(m-1)$ vertices which are closer to $u_i$ than $v$, and there is  no vertices closer to $v$ than $u_i$. So $|n_{u_i}(u_iv,Q(m,n))-n_{v}(u_iv,Q(m,n))|=n(m-1)$. Since there are $m(n-1)$ edges like $u_iv$ in $Q(m,n)$, therefore we have the result.
		\item[(ii)] 
		First consider the edge $u_iu_j$ in $K_m$.  There are $\frac{n(n-1)}{2}$ edges which are closer to $u_i$ than $u_j$, and there are  $\frac{n(n-1)}{2}$ edges closer to $u_j$ than $u_i$. So $|m_{u_i}(u_iu_j,Q(m,n))-m_{u_j}(u_ju_i,Q(m,n))|=0$. Now consider the edge $vw$ in the $i$-th $K_n$. There is no edges which are closer to $v$ than $w$, and visa versa. So $|m_{v}(vw,Q(m,n))-m_{w}(vw,Q(m,n))|=0$. Finally, consider the edge $u_iv$ in the $i$-th $K_n$. There are $\frac{n(n-1)(m-1)}{2}+\frac{m(m-1)}{2}$ edges which are closer to $u_i$ than $v$, and there is  no edges closer to $v$ than $u_i$. So $|m_{u_i}(u_iv,Q(m,n))-m_{v}(u_iv,Q(m,n))|=\frac{(m-1)}{2}(n^2-n+m)$. Since there are $m(n-1)$ edges like $u_iv$ in $Q(m,n)$, so we have the result.	\qed
		\end{enumerate} 
			\end{proof}
	
	\begin{figure}\hspace{1.02cm}
		\begin{minipage}{7.5cm}
			\psscalebox{0.45 0.45}
			{
				\begin{pspicture}(0,-13.74)(11.956668,-7.22)
				\pscircle[linecolor=black, linewidth=0.08, dimen=outer](4.1333337,-8.02){0.8}
				\pscircle[linecolor=black, linewidth=0.08, dimen=outer](11.146667,-10.566667){0.8}
				\pscircle[linecolor=black, linewidth=0.08, dimen=outer](7.96,-8.033334){0.8}
				\pscircle[linecolor=black, linewidth=0.08, dimen=outer](0.8000002,-10.54){0.8}
				\psdots[linecolor=black, dotsize=0.4](4.2400002,-8.806666)
				\psdots[linecolor=black, dotsize=0.4](10.386667,-10.486667)
				\psdots[linecolor=black, dotsize=0.4](7.786667,-8.833333)
				\psdots[linecolor=black, dotsize=0.4](1.5600002,-10.526667)
				\rput[bl](5.6933336,-10.76){$\Large{K_m}$}
				\rput[bl](0.44666687,-10.706667){$\large{K_n}$}
				\rput[bl](3.9133337,-8.066667){$\large{K_n}$}
				\rput[bl](7.7000003,-8.08){$\large{K_n}$}
				\rput[bl](10.946667,-10.78){$\large{K_n}$}
				\psellipse[linecolor=black, linewidth=0.08, dimen=outer](5.9866667,-10.426666)(4.4533334,1.78)
				\psdots[linecolor=black, dotsize=0.1](9.533334,-12.246667)
				\psdots[linecolor=black, dotsize=0.1](8.493334,-12.633333)
				\psdots[linecolor=black, dotsize=0.1](7.306667,-12.86)
				\rput[bl](4.1200004,-9.273334){$u_1$}
				\rput[bl](7.4133334,-9.3){$u_2$}
				\rput[bl](1.7733335,-10.726666){$u_m$}
				\psdots[linecolor=black, dotsize=0.1](9.286667,-8.34)
				\psdots[linecolor=black, dotsize=0.1](9.906667,-8.64)
				\psdots[linecolor=black, dotsize=0.1](10.386667,-9.14)
				\rput[bl](9.706667,-10.58){$u_i$}
				\rput[bl](6.066667,-11.96){$u_j$}
				\rput[bl](11.726666,-9.82){v}
				\rput[bl](11.226666,-11.8){w}
				\pscircle[linecolor=black, linewidth=0.08, dimen=outer](6.0,-12.94){0.8}
				\psdots[linecolor=black, dotsize=0.4](5.9866667,-12.14)
				\psdots[linecolor=black, dotsize=0.1](4.046667,-12.84)
				\psdots[linecolor=black, dotsize=0.1](3.0266669,-12.58)
				\psdots[linecolor=black, dotsize=0.1](2.0266669,-12.08)
				\rput[bl](5.766667,-13.24){$K_n$}
				\psdots[linecolor=black, dotsize=0.4](11.566667,-9.96)
				\psdots[linecolor=black, dotsize=0.4](11.286667,-11.34)
				\end{pspicture}
			}
		\end{minipage}
		\hspace{1.02cm}
		\begin{minipage}{7.5cm} 
			\psscalebox{0.45 0.45}
			{
				\begin{pspicture}(0,-4.8)(6.8027782,1.202778)
				\psline[linecolor=black, linewidth=0.08](3.4013891,-0.5986108)(1.8013892,-1.7986108)(2.6013892,-3.3986108)(4.2013893,-3.3986108)(5.001389,-1.7986108)(3.4013891,-0.5986108)(3.4013891,-0.5986108)
				\psline[linecolor=black, linewidth=0.08](3.4013891,-0.5986108)(2.6013892,0.20138916)(3.4013891,1.0013891)(4.2013893,0.20138916)(3.4013891,-0.5986108)(3.4013891,-0.5986108)
				\psline[linecolor=black, linewidth=0.08](3.4013891,1.0013891)(3.4013891,-0.5986108)(3.4013891,-0.5986108)
				\psline[linecolor=black, linewidth=0.08](2.6013892,0.20138916)(4.2013893,0.20138916)(4.2013893,0.20138916)
				\psline[linecolor=black, linewidth=0.08](5.001389,-1.7986108)(5.801389,-0.99861085)(6.601389,-1.7986108)(5.801389,-2.5986109)(5.001389,-1.7986108)(6.601389,-1.7986108)(6.601389,-1.7986108)
				\psline[linecolor=black, linewidth=0.08](5.801389,-0.99861085)(5.801389,-2.5986109)(5.801389,-2.5986109)
				\psline[linecolor=black, linewidth=0.08](4.2013893,-3.3986108)(5.401389,-3.3986108)(5.401389,-4.598611)(4.2013893,-4.598611)(4.2013893,-3.3986108)(5.401389,-4.598611)(5.401389,-4.598611)
				\psline[linecolor=black, linewidth=0.08](5.401389,-3.3986108)(4.2013893,-4.598611)(4.2013893,-4.598611)
				\psline[linecolor=black, linewidth=0.08](2.6013892,-3.3986108)(2.6013892,-4.598611)(1.4013891,-4.598611)(1.4013891,-3.3986108)(2.6013892,-3.3986108)(1.4013891,-4.598611)(1.4013891,-3.3986108)(2.6013892,-4.598611)(2.6013892,-4.598611)
				\psline[linecolor=black, linewidth=0.08](1.8013892,-1.7986108)(1.0013891,-0.99861085)(0.20138916,-1.7986108)(1.0013891,-2.5986109)(1.8013892,-1.7986108)(0.20138916,-1.7986108)(1.0013891,-0.99861085)(1.0013891,-2.5986109)(1.0013891,-2.5986109)
				\psline[linecolor=black, linewidth=0.08](3.4013891,-0.5986108)(2.6013892,-3.3986108)(5.001389,-1.7986108)(1.8013892,-1.7986108)(4.2013893,-3.3986108)(3.4013891,-0.5986108)(3.4013891,-0.5986108)
				\psdots[linecolor=black, dotstyle=o, dotsize=0.4, fillcolor=white](3.4013891,-0.5986108)
				\psdots[linecolor=black, dotstyle=o, dotsize=0.4, fillcolor=white](4.2013893,0.20138916)
				\psdots[linecolor=black, dotstyle=o, dotsize=0.4, fillcolor=white](3.4013891,1.0013891)
				\psdots[linecolor=black, dotstyle=o, dotsize=0.4, fillcolor=white](2.6013892,0.20138916)
				\psdots[linecolor=black, dotstyle=o, dotsize=0.4, fillcolor=white](5.801389,-0.99861085)
				\psdots[linecolor=black, dotstyle=o, dotsize=0.4, fillcolor=white](5.001389,-1.7986108)
				\psdots[linecolor=black, dotstyle=o, dotsize=0.4, fillcolor=white](5.801389,-2.5986109)
				\psdots[linecolor=black, dotstyle=o, dotsize=0.4, fillcolor=white](6.601389,-1.7986108)
				\psdots[linecolor=black, dotstyle=o, dotsize=0.4, fillcolor=white](1.8013892,-1.7986108)
				\psdots[linecolor=black, dotstyle=o, dotsize=0.4, fillcolor=white](1.0013891,-2.5986109)
				\psdots[linecolor=black, dotstyle=o, dotsize=0.4, fillcolor=white](0.20138916,-1.7986108)
				\psdots[linecolor=black, dotstyle=o, dotsize=0.4, fillcolor=white](1.0013891,-0.99861085)
				\psdots[linecolor=black, dotstyle=o, dotsize=0.4, fillcolor=white](1.4013891,-3.3986108)
				\psdots[linecolor=black, dotstyle=o, dotsize=0.4, fillcolor=white](1.4013891,-4.598611)
				\psdots[linecolor=black, dotstyle=o, dotsize=0.4, fillcolor=white](2.6013892,-3.3986108)
				\psdots[linecolor=black, dotstyle=o, dotsize=0.4, fillcolor=white](2.6013892,-4.598611)
				\psdots[linecolor=black, dotstyle=o, dotsize=0.4, fillcolor=white](4.2013893,-3.3986108)
				\psdots[linecolor=black, dotstyle=o, dotsize=0.4, fillcolor=white](5.401389,-3.3986108)
				\psdots[linecolor=black, dotstyle=o, dotsize=0.4, fillcolor=white](5.401389,-4.598611)
				\psdots[linecolor=black, dotstyle=o, dotsize=0.4, fillcolor=white](4.2013893,-4.598611)
				\end{pspicture}
			}
		\end{minipage}
		\caption{The graph $Q(m,n)$ and $Q(5,4)$, respectively. }\label{qmn}
	\end{figure}
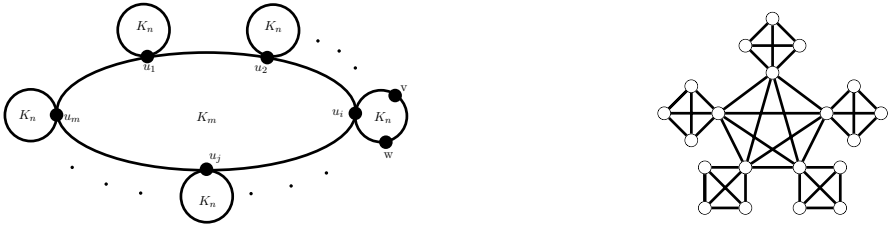

\subsection{Upper bounds for the Mostar (edge Mostar) index of polymers} 
In this subsection, we consider some special polymer graphs and present upper bounds for the Mosoar index and edge Mostar index of them.  	
	The following theorem is about the link of graphs.

	\begin{theorem} \label{thm-link}
		Let $G$ be a polymer graph with composed of 	monomers $\{G_i\}_{i=1}^k$ with respect to the vertices $\{x_i, y_i\}_{i=1}^k$. Let $G$ be the link of graphs  (see Figure \ref{link-n}).  Then,
		\begin{enumerate} 
			\item[(i)] 
					\begin{align*}
					Mo(G) \leq \sum_{i=1}^{n}Mo(G_i)+\sum_{i=1}^{n}|E(G_i)|(|V(G)|-|V(G_i)|)+\sum_{i=1}^{n-1}\Bigl|
		\sum_{t=1}^{i}|V(G_t)|-\sum_{t=i+1}^{n}|V(G_t)|\Bigr|.
		\end{align*}
		
		\item[(ii)] 
		\begin{align*}
		Mo_e(G) \leq \sum_{i=1}^{n}Mo_e(G_i)+\sum_{i=1}^{n}|E(G_i)|(|E(G)|-|E(G_i)|)+\sum_{i=1}^{n-1}
		\Bigl|\sum_{t=1}^{i}|E(G_t)|-\sum_{t=i+1}^{n}|E(G_t)|\Bigr|.
		\end{align*}
		\end{enumerate} 
			\end{theorem}

	\begin{proof}
		\begin{enumerate} 
			\item[(i)] 
			Consider the graph $G_i$ (Figure \ref{link-n}) and let $n_u(uv,G_i)$ be the number of vertices of $G_i$ closer to $u$ than $v$ in $G_i$. By the definition of Mostar index, we have:
		
		\begin{align*}
		Mo(G)&= \sum_{uv\in E(G)}^{}|n_u(uv,G)-n_v(uv,G)|\\
		&= \sum_{i=1}^{n}\sum_{uv\in E(G_i)}^{}|n_u(uv,G_i)-n_v(uv,G_i)|\\
		\end{align*}
		
		\begin{align*}
		&\quad +\sum_{i=1}^{n-1}\sum_{y_ix_{i+1}\in E(G)}^{}|n_{y_i}(y_ix_{i+1},G)-n_{x_{i+1}}(y_ix_{i+1},G)|\\
		&=\sum_{i=1}^{n}\sum_{uv\in E(G_i), d(u,x_i)<d(v,x_i),  d(u,y_i)<d(v,y_i)}^{}|n_u(uv,G_i)-n_v(uv,G_i)|\\
		&\quad +\sum_{i=1}^{n}\sum_{uv\in E(G_i),d(u,x_i)<d(v,x_i), d(v,y_i)<d(u,y_i)}^{}|n_u(uv,G_i)-n_v(uv,G_i)|\\
		&\quad +\sum_{i=1}^{n}\sum_{uv\in E(G_i), d(u,x_i)<d(v,x_i),  d(u,y_i)=d(v,y_i)}^{}|n_u(uv,G_i)-n_v(uv,G_i)|\\
		&\quad +\sum_{i=1}^{n}\sum_{uv\in E(G_i), d(u,x_i)=d(v,x_i),  d(u,y_i)<d(v,y_i)}^{}|n_u(uv,G_i)-n_v(uv,G_i)|\\
		&\quad +\sum_{i=1}^{n}\sum_{uv\in E(G_i), d(u,x_i)=d(v,x_i),  d(u,y_i)=d(v,y_i)}^{}|n_u(uv,G_i)-n_v(uv,G_i)|\\
		&\quad +\sum_{i=1}^{n-1}\sum_{y_ix_{i+1}\in E(G)}^{}|n_{y_i}(y_ix_{i+1},G)-n_{x_{i+1}}(y_ix_{i+1},G)|\\
		&=\sum_{i=1}^{n}\sum_{uv\in E(G_i), d(u,x_i)<d(v,x_i),  d(u,y_i)<d(v,y_i)}^{} \Bigl| n_u'(uv,G_i)+|V(G)-V(G_i)|-n_v'(uv,G_i)\Bigr|\\
		&\quad +\sum_{i=1}^{n}\sum_{uv\in E(G_i),d(u,x_i)<d(v,x_i), d(v,y_i)<d(u,y_i)}^{} \\
		&\quad\quad\quad\quad \Bigl| n_u'(uv,G_i)+\sum_{t=1}^{i}|V(G_t)|-n_v'(uv,G_i)-\sum_{t=i+1}^{n}|V(G_t)|\Bigr|\\
		&\quad +\sum_{i=1}^{n}\sum_{uv\in E(G_i), d(u,x_i)<d(v,x_i),  d(u,y_i)=d(v,y_i)}^{}\Bigl| n_u'(uv,G_i)+\sum_{t=1}^{i}|V(G_t)|-n_v'(uv,G_i)\Bigr|\\
		&\quad +\sum_{i=1}^{n}\sum_{uv\in E(G_i), d(u,x_i)=d(v,x_i),  d(u,y_i)<d(v,y_i)}^{}\Bigl| n_u'(uv,G_i)-n_v'(uv,G_i)-\sum_{t=i+1}^{n}|V(G_t)|\Bigr|\\
		&\quad +\sum_{i=1}^{n}\sum_{uv\in E(G_i), d(u,x_i)=d(v,x_i),  d(u,y_i)=d(v,y_i)}^{}|n_u'(uv,G_i)-n_v'(uv,G_i)|\\
		&\quad +\sum_{i=1}^{n-1}\Bigl| \sum_{t=1}^{i}|V(G_t)| - \sum_{t=i+1}^{n}|V(G_t)|  \Bigr|\\
		&\leq\sum_{i=1}^{n}\sum_{uv\in E(G_i), d(u,x_i)<d(v,x_i),  d(u,y_i)<d(v,y_i)}^{} | n_u'(uv,G_i)-n_v'(uv,G_i)|\\
		\end{align*}
		
		\begin{align*}
		&\quad+\sum_{i=1}^{n}\sum_{uv\in E(G_i), d(u,x_i)<d(v,x_i),  d(u,y_i)<d(v,y_i)}^{}|V(G)-V(G_i)|\\
		&\quad +\sum_{i=1}^{n}\sum_{uv\in E(G_i),d(u,x_i)<d(v,x_i), d(v,y_i)<d(u,y_i)}^{} | n_u'(uv,G_i)-n_v'(uv,G_i)|\\
		&\quad+\sum_{i=1}^{n}\sum_{uv\in E(G_i), d(u,x_i)<d(v,x_i),  d(u,y_i)<d(v,y_i)}^{}|V(G)-V(G_i)|\\
		&\quad +\sum_{i=1}^{n}\sum_{uv\in E(G_i), d(u,x_i)<d(v,x_i),  d(u,y_i)=d(v,y_i)}^{} | n_u'(uv,G_i)-n_v'(uv,G_i)|\\
		&\quad+\sum_{i=1}^{n}\sum_{uv\in E(G_i), d(u,x_i)<d(v,x_i),  d(u,y_i)<d(v,y_i)}^{}|V(G)-V(G_i)|\\
		&\quad +\sum_{i=1}^{n}\sum_{uv\in E(G_i), d(u,x_i)=d(v,x_i),  d(u,y_i)<d(v,y_i)}^{} | n_u'(uv,G_i)-n_v'(uv,G_i)|\\
		&\quad+\sum_{i=1}^{n}\sum_{uv\in E(G_i), d(u,x_i)<d(v,x_i),  d(u,y_i)<d(v,y_i)}^{}|V(G)-V(G_i)|\\
		&\quad +\sum_{i=1}^{n}\sum_{uv\in E(G_i), d(u,x_i)=d(v,x_i),  d(u,y_i)=d(v,y_i)}^{}|n_u'(uv,G_i)-n_v'(uv,G_i)|\\
		&\quad +\sum_{i=1}^{n-1}\Bigl| \sum_{t=1}^{i}|V(G_t)| - \sum_{t=i+1}^{n}|V(G_t)|  \Bigr|\\
		&= \sum_{i=1}^{n}Mo(G_i)+\sum_{i=1}^{n}|E(G_i)|( |V(G)|-|V(G_i)|)+\sum_{i=1}^{n-1}\Bigl|
		\sum_{t=1}^{i}|V(G_t)|-\sum_{t=i+1}^{n}|V(G_t)|\Bigr|.
		\end{align*}
		
		Therefore, we have the result.	
		\item[(ii)] The proof is similar to Part (i). \qed
		\end{enumerate} 
	\end{proof}

	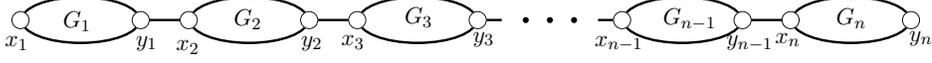
\begin{figure}
		\begin{center}
			\psscalebox{0.8 0.8}
			{
				\begin{pspicture}(0,-4.08)(15.436667,-3.12)
				\psellipse[linecolor=black, linewidth=0.04, dimen=outer](1.2533334,-3.52)(1.0,0.4)
				\psellipse[linecolor=black, linewidth=0.04, dimen=outer](4.0533333,-3.52)(1.0,0.4)
				\psellipse[linecolor=black, linewidth=0.04, dimen=outer](6.853334,-3.52)(1.0,0.4)
				\psellipse[linecolor=black, linewidth=0.04, dimen=outer](11.253334,-3.52)(1.0,0.4)
				\psellipse[linecolor=black, linewidth=0.04, dimen=outer](14.053333,-3.52)(1.0,0.4)
				\psline[linecolor=black, linewidth=0.04](2.2533333,-3.52)(3.0533335,-3.52)(3.0533335,-3.52)
				\psline[linecolor=black, linewidth=0.04](5.0533333,-3.52)(5.8533335,-3.52)(5.8533335,-3.52)
				\psline[linecolor=black, linewidth=0.04](12.253333,-3.52)(13.053333,-3.52)(13.053333,-3.52)
				\psdots[linecolor=black, dotstyle=o, dotsize=0.3, fillcolor=white](2.2533333,-3.52)
				\psdots[linecolor=black, dotstyle=o, dotsize=0.3, fillcolor=white](0.25333345,-3.52)
				\psdots[linecolor=black, dotstyle=o, dotsize=0.3, fillcolor=white](3.0533335,-3.52)
				\psdots[linecolor=black, dotstyle=o, dotsize=0.3, fillcolor=white](5.0533333,-3.52)
				\psdots[linecolor=black, dotstyle=o, dotsize=0.3, fillcolor=white](5.8533335,-3.52)
				\psdots[linecolor=black, dotstyle=o, dotsize=0.3, fillcolor=white](12.253333,-3.52)
				\psdots[linecolor=black, dotstyle=o, dotsize=0.3, fillcolor=white](13.053333,-3.52)
				\psdots[linecolor=black, dotstyle=o, dotsize=0.3, fillcolor=white](15.053333,-3.52)
				\rput[bl](0.0,-4.0133333){$x_1$}
				\rput[bl](2.8400002,-4.08){$x_2$}
				\rput[bl](5.5866666,-4.0){$x_3$}
				\rput[bl](2.1733334,-4.0266666){$y_1$}
				\rput[bl](4.92,-4.0266666){$y_2$}
				\rput[bl](7.7733335,-3.96){$y_3$}
				\rput[bl](0.9600001,-3.7066667){$G_1$}
				\rput[bl](3.8000002,-3.6666667){$G_2$}
				\rput[bl](6.64,-3.64){$G_3$}
				\psline[linecolor=black, linewidth=0.04](8.253333,-3.52)(7.8533335,-3.52)(7.8533335,-3.52)
				\psline[linecolor=black, linewidth=0.04](9.853333,-3.52)(10.253333,-3.52)(10.253333,-3.52)
				\psdots[linecolor=black, dotstyle=o, dotsize=0.3, fillcolor=white](7.8533335,-3.52)
				\psdots[linecolor=black, dotstyle=o, dotsize=0.3, fillcolor=white](10.253333,-3.52)
				\psdots[linecolor=black, dotsize=0.1](8.653334,-3.52)
				\psdots[linecolor=black, dotsize=0.1](9.053333,-3.52)
				\psdots[linecolor=black, dotsize=0.1](9.453334,-3.52)
				\rput[bl](12.8,-3.96){$x_n$}
				\rput[bl](15.026667,-3.9466667){$y_n$}
				\rput[bl](11.986667,-4.0266666){$y_{n-1}$}
				\rput[bl](10.933333,-3.68){$G_{n-1}$}
				\rput[bl](9.8,-4.04){$x_{n-1}$}
				\rput[bl](13.8133335,-3.6533334){$G_n$}
				\end{pspicture}
			}
		\end{center}
		\caption{Link of $n$ graphs $G_1,G_2, \ldots , G_n$} \label{link-n}
	\end{figure}

	By the same argument similar to  the proof of the Theorem \ref{thm-link}, we have:
	
	\begin{theorem} 
		Let $G_1,G_2, \ldots , G_n$ be a finite sequence of pairwise disjoint connected graphs and let
		$x_i,y_i \in V(G_i)$. Let $C(G_1,...,G_n)$ be the chain of graphs $\{G_i\}_{i=1}^n$ with respect to the vertices $\{x_i, y_i\}_{i=1}^k$ which obtained by identifying the vertex $y_i$ with the vertex $x_{i+1}$ for $i=1,2,\ldots,n-1$ (Figure \ref{chain-n}). Then,
		\begin{enumerate}
			\item [(i)]
				$Mo(C(G_1,...,G_n))\leq \sum_{i=1}^{n}Mo(G_i)+\sum_{i=1}^{n}|E(G_i)|(|V(G)|-|V(G_i)|).$
		\item[(ii)] 
			$Mo_e(C(G_1,...,G_n))\leq \sum_{i=1}^{n}Mo_e(G_i)+\sum_{i=1}^{n}|E(G_i)|(|E(G)|-|E(G_i)|).$
			\end{enumerate} 
			\end{theorem}

	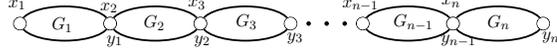
\begin{figure}
		\begin{center}
			\psscalebox{0.6 0.6}
			{
				\begin{pspicture}(0,-3.9483333)(12.236668,-2.8316667)
				\psellipse[linecolor=black, linewidth=0.04, dimen=outer](1.2533334,-3.4416668)(1.0,0.4)
				\psellipse[linecolor=black, linewidth=0.04, dimen=outer](3.2533333,-3.4416668)(1.0,0.4)
				\psellipse[linecolor=black, linewidth=0.04, dimen=outer](5.2533336,-3.4416668)(1.0,0.4)
				\psellipse[linecolor=black, linewidth=0.04, dimen=outer](8.853333,-3.4416668)(1.0,0.4)
				\psellipse[linecolor=black, linewidth=0.04, dimen=outer](10.853333,-3.4416668)(1.0,0.4)
				\psdots[linecolor=black, fillstyle=solid, dotstyle=o, dotsize=0.3, fillcolor=white](2.2533333,-3.4416666)
				\psdots[linecolor=black, fillstyle=solid, dotstyle=o, dotsize=0.3, fillcolor=white](0.25333345,-3.4416666)
				\psdots[linecolor=black, fillstyle=solid, dotstyle=o, dotsize=0.3, fillcolor=white](2.2533333,-3.4416666)
				\psdots[linecolor=black, fillstyle=solid, dotstyle=o, dotsize=0.3, fillcolor=white](4.2533336,-3.4416666)
				\psdots[linecolor=black, fillstyle=solid, dotstyle=o, dotsize=0.3, fillcolor=white](4.2533336,-3.4416666)
				\psdots[linecolor=black, fillstyle=solid, dotstyle=o, dotsize=0.3, fillcolor=white](9.853333,-3.4416666)
				\psdots[linecolor=black, fillstyle=solid, dotstyle=o, dotsize=0.3, fillcolor=white](9.853333,-3.4416666)
				\psdots[linecolor=black, fillstyle=solid, dotstyle=o, dotsize=0.3, fillcolor=white](11.853333,-3.4416666)
				\rput[bl](0.0,-3.135){$x_1$}
				\rput[bl](2.0400002,-3.2016668){$x_2$}
				\rput[bl](3.9866667,-3.1216667){$x_3$}
				\rput[bl](2.1733334,-3.9483335){$y_1$}
				\rput[bl](4.12,-3.9483335){$y_2$}
				\rput[bl](6.1733336,-3.8816667){$y_3$}
				\rput[bl](0.9600001,-3.6283333){$G_1$}
				\rput[bl](3.0,-3.5883334){$G_2$}
				\rput[bl](5.04,-3.5616667){$G_3$}
				\psdots[linecolor=black, fillstyle=solid, dotstyle=o, dotsize=0.3, fillcolor=white](6.2533336,-3.4416666)
				\psdots[linecolor=black, fillstyle=solid, dotstyle=o, dotsize=0.3, fillcolor=white](7.8533335,-3.4416666)
				\psdots[linecolor=black, dotsize=0.1](6.6533337,-3.4416666)
				\psdots[linecolor=black, dotsize=0.1](7.0533333,-3.4416666)
				\psdots[linecolor=black, dotsize=0.1](7.4533334,-3.4416666)
				\rput[bl](9.6,-3.0816667){$x_n$}
				\rput[bl](11.826667,-3.8683333){$y_n$}
				\rput[bl](9.586667,-3.9483335){$y_{n-1}$}
				\rput[bl](8.533334,-3.6016667){$G_{n-1}$}
				\rput[bl](7.4,-3.1616666){$x_{n-1}$}
				\rput[bl](10.613334,-3.575){$G_n$}
				\end{pspicture}
			}
		\end{center}
		\caption{Chain of $n$ graphs $G_1,G_2, \ldots , G_n$} \label{chain-n}
	\end{figure}

	With similar argument to the proof of the Theorem \ref{thm-link}, we have:

	\begin{theorem} 
		Let $G_1,G_2, \ldots , G_n$ be a finite sequence of pairwise disjoint connected graphs and let
		$x_i \in V(G_i)$. Let $B(G_1,...,G_n)$ be the bouquet of graphs $\{G_i\}_{i=1}^n$ with respect to the vertices $\{x_i\}_{i=1}^n$ and obtained by identifying the vertex $x_i$ of the graph $G_i$ with $x$ (see Figure \ref{bouquet-n}). Then,
		\begin{enumerate}
		 \item[(i)] 
						$$Mo(B(G_1,...,G_n)) \leq \sum_{i=1}^{n}Mo(G_i)+\sum_{i=1}^{n}|E(G_i)|(|V(G)|-|V(G_i)|).$$
						\item[(ii)] 
							$$Mo_e(B(G_1,...,G_n)) \leq \sum_{i=1}^{n}Mo_e(G_i)+\sum_{i=1}^{n}|E(G_i)|(|E(G)|-|E(G_i)|).$$
					\end{enumerate}	
					\end{theorem}
	
	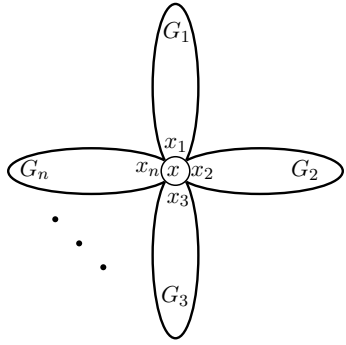
\begin{figure}
		\begin{center}
			\psscalebox{0.8 0.8}
			{
				\begin{pspicture}(0,-6.76)(5.6,-1.16)
				\rput[bl](2.6133332,-3.64){$x_1$}
				\rput[bl](3.0533333,-4.0933332){$x_2$}
				\rput[bl](2.6533334,-4.5466666){$x_3$}
				\rput[bl](2.5866666,-1.8133334){$G_1$}
				\rput[bl](4.72,-4.1066666){$G_2$}
				\rput[bl](2.56,-6.2){$G_3$}
				\rput[bl](2.1333334,-4.04){$x_n$}
				\rput[bl](0.21333334,-4.0933332){$G_n$}
				\psellipse[linecolor=black, linewidth=0.04, dimen=outer](1.4,-3.96)(1.4,0.4)
				\psellipse[linecolor=black, linewidth=0.04, dimen=outer](2.8,-2.56)(0.4,1.4)
				\psellipse[linecolor=black, linewidth=0.04, dimen=outer](4.2,-3.96)(1.4,0.4)
				\psellipse[linecolor=black, linewidth=0.04, dimen=outer](2.8,-5.36)(0.4,1.4)
				\psdots[linecolor=black, dotsize=0.1](0.8,-4.76)
				\psdots[linecolor=black, dotsize=0.1](1.2,-5.16)
				\psdots[linecolor=black, dotsize=0.1](1.6,-5.56)
				\psdots[linecolor=black, dotstyle=o, dotsize=0.5, fillcolor=white](2.8,-3.96)
				\rput[bl](2.6533334,-4.04){$x$}
				\end{pspicture}
			}
		\end{center}
		\caption{Bouquet of $n$ graphs $G_1,G_2, \ldots , G_n$ and $x_1=x_2=\ldots=x_n=x$} \label{bouquet-n}
	\end{figure}

	\begin{theorem} 
		Let $G_1,G_2, \ldots , G_n$ be a finite sequence of pairwise disjoint connected graphs and let
		$x_i \in V(G_i)$. Let $G$ be the circuit of graphs $\{G_i\}_{i=1}^n$ with respect to the vertices $\{x_i\}_{i=1}^n$ and obtained by identifying the vertex $x_i$ of the graph $G_i$ with the $i$-th vertex of the
		cycle graph $C_n$ (Figure \ref{circuit-n}). Then,
		\begin{align*}	
		Mo(G) &\leq \sum_{i=1}^{n}Mo(G_i)+\sum_{i=1}^{n}|E(G_i)|(|V(G)|-|V(G_i)|)\\
		&\quad+\left\{
		\begin{array}{ll}
		{\displaystyle
			n\sum_{i=1}^{t} \Bigl| |V(G_i)|-|V(G_{t+i})|\Bigr|}&
		\quad\mbox{if $n=2t$, }\\[15pt]
		{\displaystyle
			(n-1)|V(G)|}&
		\quad\mbox{if $n=2t-1$.}
		\end{array}
		\right.	\\
		\end{align*}
	\end{theorem}

	\begin{proof}
		First consider the edge $x_1x_n$. There are two cases, $n$ is even or odd. If $n=2t$ for some  $t\in \mathbb{N}$, then, the vertices in the graphs $G_1,G_2,G_3,\ldots,G_t$ are closer to $x_1$ than $x_n$, and the rest are closer to $x_n$ than $x_1$. So
		\begin{align*}
		|n_{x_1}(x_1x_n,G)-n_{x_n}(x_1x_n,G)|&=\Bigl|\sum_{i=1}^{t}|V(G_i)|-\sum_{i=1}^{t}|V(G_{t+i})|\Bigr|\\
		&\leq \sum_{i=1}^{t} \Bigl| |V(G_i)|-|V(G_{t+i})|\Bigr|. 
		\end{align*}
		It is easy to check that the same happens for $x_ix_{i+1}$ for all $1\leq i \leq n-1$.
	\\		If $n=2t-1$ for some $t\in \mathbb{N}$, then, the vertices in the graphs $G_1,G_2,G_3,\ldots,G_{t-1}$ are closer to $x_1$ than $x_n$, and the vertices in the graphs $G_{t+1},G_{t+2},G_{t+3},\ldots,G_n$ are closer to $x_n$ than $x_1$. The vertices in the graph $G_{t}$ have the same distance to $x_1$ and $x_n$. So 
		\begin{align*}
		|n_{x_1}(x_1x_n,G)-n_{x_n}(x_1x_n,G)|&=\Bigl|\sum_{i=1}^{t-1}|V(G_i)|-\sum_{i=1}^{t-1}|V(G_{t+i})|\Bigr|\\
		&\leq |V(G_1)|+|V(G_2)|+\ldots+|V(G_{t-1})|\\
		&\quad +|V(G_{t+1})|+|V(G_{t+2})|+\ldots+|V(G_n)|\\
		&= |V(G)|-|V(G_t)|.
		\end{align*}
				It is easy to check that $|n_{x_1}(x_1x_2,G)-n_{x_2}(x_1x_2,G)|\leq |V(G)|-|V(G_{t+1})| $, and this continues.
				Now we consider the edge $uv\in G_i$. There are two cases, first $u$ is closer to $x_i$ than $v$, and second they have the same distance to $x_i$. Let $n_u'(uv,G_i)$ be the number of vertices of $G_i$ closer to $u$ than $v$ in $G_i$. Then by the definition of Mostar index, we have:
		
		\begin{align*}
		Mo(G)&= \sum_{uv\in E(G)}^{}|n_u(uv,G)-n_v(uv,G)|\\
		&= \sum_{i=1}^{n}\sum_{uv\in E(G_i)}^{}|n_u(uv,G_i)-n_v(uv,G_i)|\\
		\end{align*}

		\begin{align*}
		&\quad +\sum_{i=1}^{n-1}\sum_{x_ix_{i+1}\in E(G)}^{}|n_{x_i}(x_ix_{i+1},G)-n_{x_{i+1}}(x_ix_{i+1},G)|\\
		&\quad + |n_{x_1}(x_1x_n,G)-n_{x_n}(x_1x_n,G)|\\
		&=\sum_{i=1}^{n}\sum_{uv\in E(G_i), d(u,x_i)<d(v,x_i)}^{}|n_u(uv,G_i)-n_v(uv,G_i)|\\
		&\quad +\sum_{i=1}^{n}\sum_{uv\in E(G_i),d(u,x_i)=d(v,x_i)}^{}|n_u(uv,G_i)-n_v(uv,G_i)|\\
		&\quad +\sum_{i=1}^{n-1}\sum_{x_ix_{i+1}\in E(G)}^{}|n_{x_i}(x_ix_{i+1},G)-n_{x_{i+1}}(x_ix_{i+1},G)|\\
		&\quad + |n_{x_1}(x_1x_n,G)-n_{x_n}(x_1x_n,G)|\\
		&=\sum_{i=1}^{n}\sum_{uv\in E(G_i), d(u,x_i)<d(v,x_i)}^{}\Bigl| n_u'(uv,G_i)+|V(G)-V(G_i)|-n_v'(uv,G_i)\Bigr|\\
		&\quad +\sum_{i=1}^{n}\sum_{uv\in E(G_i),d(u,x_i)=d(v,x_i)}^{}|n_u'(uv,G_i)-n_v'(uv,G_i)|\\
		&\quad +\sum_{i=1}^{n-1}\sum_{x_ix_{i+1}\in E(G)}^{}|n_{x_i}(x_ix_{i+1},G)-n_{x_{i+1}}(x_ix_{i+1},G)|\\
		&\quad + |n_{x_1}(x_1x_n,G)-n_{x_n}(x_1x_n,G)|\\
		&\leq\sum_{i=1}^{n}\sum_{uv\in E(G_i),d(u,x_i)<d(v,x_i)}^{}|n_u'(uv,G_i)-n_v'(uv,G_i)|\\ 
		&\quad +\sum_{i=1}^{n}\sum_{uv\in E(G_i),d(u,x_i)<d(v,x_i)}^{}|V(G)-V(G_i)|\\
		&\quad +\sum_{i=1}^{n}\sum_{uv\in E(G_i),d(u,x_i)=d(v,x_i)}^{}|n_u'(uv,G_i)-n_v'(uv,G_i)|\\ 	
		&\quad +\sum_{i=1}^{n}\sum_{uv\in E(G_i),d(u,x_i)=d(v,x_i)}^{}|V(G)-V(G_i)|\\	
		&\quad+\left\{
		\begin{array}{ll}
		{\displaystyle
			n\sum_{i=1}^{t} \Bigl| |V(G_i)|-|V(G_{t+i})|\Bigr|}&
		\quad\mbox{if $n=2t$, }\\[15pt]
		{\displaystyle
			(n-1)|V(G)|}&
		\quad\mbox{if $n=2t-1$,}
		\end{array}
		\right.	\\
		\end{align*}
		
		\begin{align*}
		&= \sum_{i=1}^{n}Mo(G_i)+\sum_{i=1}^{n}|E(G_i)|(|V(G)|-|V(G_i)|)\\
		&\quad+\left\{
		\begin{array}{ll}
		{\displaystyle
			n\sum_{i=1}^{t} \Bigl| |V(G_i)|-|V(G_{t+i})|\Bigr|}&
		\quad\mbox{if $n=2t$, }\\[15pt]
		{\displaystyle
			(n-1)|V(G)|}&
		\quad\mbox{if $n=2t-1$.}
		\end{array}
		\right.	\\
		\end{align*}
		Therefore, we have the result.	\qed
	\end{proof}		
	
	Similarly, we have the following result for the edge Mostar index of circuit of graphs:
	
	\begin{theorem} 
		Let $G_1,G_2, \ldots , G_n$ be a finite sequence of pairwise disjoint connected graphs and let
		$x_i \in V(G_i)$. Let $G$ be the circuit of graphs $\{G_i\}_{i=1}^n$ with respect to the vertices $\{x_i\}_{i=1}^n$ and obtained by identifying the vertex $x_i$ of the graph $G_i$ with the $i$-th vertex of the
		cycle graph $C_n$ (Figure \ref{circuit-n}). Then,
		\begin{align*}	
		Mo_e(G) &\leq \sum_{i=1}^{n}Mo_e(G_i)+\sum_{i=1}^{n}|E(G_i)|(|E(G)|-|E(G_i)|)\\
		&\quad+\left\{
		\begin{array}{ll}
		{\displaystyle
			n\sum_{i=1}^{t} \Bigl| |E(G_i)|-|E(G_{t+i})|\Bigr|}&
		\quad\mbox{if $n=2t$, }\\[15pt]
		{\displaystyle
			(n-1)|E(G)|}&
		\quad\mbox{if $n=2t-1$.}
		\end{array}
		\right.	\\
		\end{align*}
	\end{theorem}

	\begin{figure}
		\begin{center}
			\psscalebox{0.85 0.85}
			{
				\begin{pspicture}(0,-7.6)(5.6,-2.0)
				\rput[bl](2.6533334,-4.48){$x_1$}
				\rput[bl](3.0533333,-4.92){$x_2$}
				\rput[bl](2.5733333,-5.4266667){$x_3$}
				\rput[bl](2.6,-3.1733334){$G_1$}
				\rput[bl](4.2933335,-4.9866667){$G_2$}
				\rput[bl](2.6133332,-6.7733335){$G_3$}
				\rput[bl](2.1733334,-4.9466667){$x_n$}
				\rput[bl](0.73333335,-4.9333334){$G_n$}
				\psellipse[linecolor=black, linewidth=0.04, dimen=outer](1.0,-4.8)(1.0,0.4)
				\psellipse[linecolor=black, linewidth=0.04, dimen=outer](4.6,-4.8)(1.0,0.4)
				\psellipse[linecolor=black, linewidth=0.04, dimen=outer](2.8,-3.0)(0.4,1.0)
				\psellipse[linecolor=black, linewidth=0.04, dimen=outer](2.8,-6.6)(0.4,1.0)
				\psline[linecolor=black, linewidth=0.04](2.0,-4.8)(2.8,-4.0)(3.6,-4.8)(2.8,-5.6)(2.8,-5.6)
				\psdots[linecolor=black, fillstyle=solid, dotstyle=o, dotsize=0.3, fillcolor=white](2.8,-4.0)
				\psdots[linecolor=black, fillstyle=solid, dotstyle=o, dotsize=0.3, fillcolor=white](3.6,-4.8)
				\psline[linecolor=black, linewidth=0.04, linestyle=dotted, dotsep=0.10583334cm](2.8,-5.6)(2.0,-4.8)(2.0,-4.8)
				\psdots[linecolor=black, fillstyle=solid, dotstyle=o, dotsize=0.3, fillcolor=white](2.0,-4.8)
				\psdots[linecolor=black, fillstyle=solid, dotstyle=o, dotsize=0.3, fillcolor=white](2.8,-5.6)
				\end{pspicture}
			}
		\end{center}
		\caption{Circuit of $n$ graphs $G_1,G_2, \ldots , G_n$} \label{circuit-n}
	\end{figure}
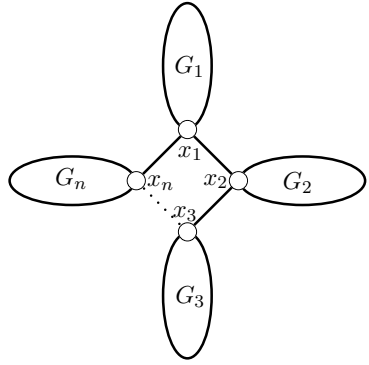

	\subsection{Lower  bounds for the Mostar (edge Mostar) index of polymers} 
	In this subsection, we consider some special polymer graphs and present lower bounds for the Mosoar index and the edge Mostar index of them.

\begin{theorem} \label{prop-lower}
	Let $G$ be a link of two graphs $G_1$ and $G_2$ with  respect to the vertices $x,y$. Then,
	\begin{enumerate} 
		\item[(i)]$MO(G)> MO(G_1) + MO(G_2) + \Bigl| |V(G_1)|-|V(G_2)|   \Bigr|.$
		\item[(ii)]$MO_e(G)> MO_e(G_1) + MO_e(G_2) + \Bigl| |E(G_1)|-|E(G_2)|   \Bigr|.$
	\end{enumerate} 
\end{theorem}

\begin{proof}
	\begin{enumerate} 
		\item[(i)] Let $n_u'(uv,G_i)$ be the number of vertices of $G_i$ closer to $u$ than $v$ in $G_i$ for $i=1,2$. By the definition of Mostar index, we have:
		
		\begin{align*}
		Mo(G)&= \sum_{uv\in E(G)}^{}|n_u(uv,G)-n_v(uv,G)|\\
		&=\sum_{uv\in E(G_1)}^{}|n_u(uv,G)-n_v(uv,G)|\\
		&\quad +\sum_{uv\in E(G_2)}^{}|n_u(uv,G)-n_v(uv,G)|\\
		&\quad + |n_{x}(xy,G)-n_{y}(xy,G)|\\
		&=\sum_{uv\in E(G_1), d(u,x)<d(v,x)}^{}|n_u(uv,G)-n_v(uv,G)|\\
		&\quad +\sum_{uv\in E(G_1),d(u,x)=d(v,x)}^{}|n_u(uv,G)-n_v(uv,G)|\\
		&\quad +\sum_{uv\in E(G_2), d(u,x)<d(v,x)}^{}|n_u(uv,G)-n_v(uv,G)|\\
		&\quad +\sum_{uv\in E(G_2),d(u,x)=d(v,x)}^{}|n_u(uv,G)-n_v(uv,G)|\\
		&\quad + |n_{x}(xy,G)-n_{y}(xy,G)|\\
		&=\sum_{uv\in E(G_1), d(u,x)<d(v,x)}^{}\Bigl| n_u'(uv,G_1) + |V(G_2)| -n_v'(uv,G-1)\Bigr|\\
		&\quad +\sum_{uv\in E(G_1),d(u,x)=d(v,x)}^{}|n_u'(uv,G_1)-n_v'(uv,G_1)|\\
		&\quad +\sum_{uv\in E(G_2), d(u,x)<d(v,x)}^{}\Bigl| n_u'(uv,G_2)+|V(G_1)|-n_v'(uv,G_2)\Bigr|\\
		&\quad +\sum_{uv\in E(G_2),d(u,x)=d(v,x)}^{}|n_u'(uv,G_2)-n_v'(uv,G_2)|\\
		\end{align*}

		\begin{align*}
		&\quad + \Bigl||V(G_1)|-|V(G_2)| \Bigr|\\
		&>\sum_{uv\in E(G_1), d(u,x)<d(v,x)}^{}| n_u'(uv,G_1)  -n_v'(uv,G_1)|\\
		&\quad +\sum_{uv\in E(G_1),d(u,x)=d(v,x)}^{}|n_u'(uv,G_1)-n_v'(uv,G_1)|\\
		&\quad +\sum_{uv\in E(G_2), d(u,x)<d(v,x)}^{}| n_u'(uv,G_2)-n_v'(uv,G_2)|\\
		&\quad +\sum_{uv\in E(G_2),d(u,x)=d(v,x)}^{}|n_u'(uv,G_2)-n_v'(uv,G_2)|\\
		&\quad + \Bigl||V(G_1)|-|V(G_2)| \Bigr|\\
		&= MO(G_1) + MO(G_2) + \Bigl| |V(G_1)|-|V(G_2)|   \Bigr|.\\
		\end{align*}
		
		\item[(ii)] The proof is similar to the proof of Part (i).  	\qed
	\end{enumerate}
\end{proof}

As an immediate result of the Theorem \ref{prop-lower}, we have:

\begin{theorem}
	Let $G$ be a polymer graph with composed of 	monomers $\{G_i\}_{i=1}^k$ with respect to the vertices $\{x_i, y_i\}_{i=1}^k$. Let $G$ be the link of graphs  (see Figure \ref{link-n}).  Then,
	\begin{enumerate} 
		\item[(i)] 
		$$Mo(G) > \sum_{i=1}^{n}Mo(G_i)+\sum_{t=1}^{n-1} \Bigl| |V(G)- \bigcup_{i=1}^{t}V(G_{i})| -|V(G_t)| \Bigr|.$$
		
		\item[(ii)] 
		$$Mo_e(G) > \sum_{i=1}^{n}Mo_e(G_i)+\sum_{t=1}^{n-1}  \Bigl| |E(G)- \bigcup_{i=1}^{t}E(G_{i})| -|E(G_t)| \Bigr|.$$
	\end{enumerate} 
\end{theorem}

	\section{Chemical applications}
	
	In this section, we obtain the Mostar index and the edge-Mostar index of
	families of graphs that are of importance in chemistry.

	\begin{theorem}
Let $T_n$ be the chain triangular graph  of order $n$. Then for every $n\geq 2$, and $k \geq 1$, we have:
\begin{enumerate}
	\item [(i)]
 \[
 	Mo(T_n)=\left\{
  	\begin{array}{ll}
  	{\displaystyle
  		12k^2-4k}&
  		\quad\mbox{if $n=2k$, }\\[15pt]
  		{\displaystyle
  			12k^2+8k}&
  			\quad\mbox{if $n=2k+1$.}
  				  \end{array}
  					\right.	
  					\]
\item[(ii)] 
 \[
 Mo_e(T_n)=\left\{
 \begin{array}{ll}
 {\displaystyle
 	18k^2-6k}&
 \quad\mbox{if $n=2k$, }\\[15pt]
 {\displaystyle
 	18k^2+12k}&
 \quad\mbox{if $n=2k+1$.}
 \end{array}
 \right.	
 \]
\end{enumerate} 
\end{theorem}

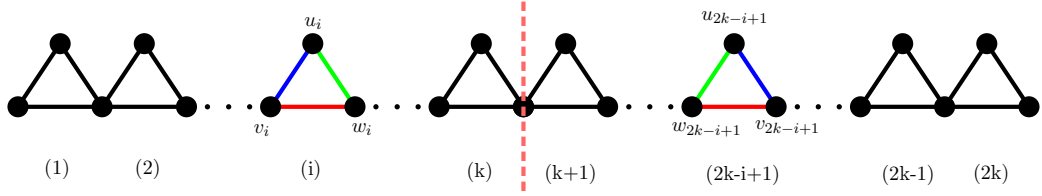
\begin{figure}
\begin{center}
\psscalebox{0.7 0.7}
{
\begin{pspicture}(0,-8.22)(19.59423,-4.56)
\definecolor{colour0}{rgb}{1.0,0.4,0.4}
\psdots[linecolor=black, dotsize=0.4](0.9971154,-5.37)
\psdots[linecolor=black, dotsize=0.4](0.1971154,-6.57)
\psdots[linecolor=black, dotsize=0.4](1.7971154,-6.57)
\psdots[linecolor=black, dotsize=0.4](2.5971155,-5.37)
\psdots[linecolor=black, dotsize=0.4](3.3971155,-6.57)
\psdots[linecolor=black, dotsize=0.1](3.7971153,-6.57)
\psdots[linecolor=black, dotsize=0.1](4.1971154,-6.57)
\psdots[linecolor=black, dotsize=0.1](4.5971155,-6.57)
\psdots[linecolor=black, dotsize=0.1](6.9971156,-6.57)
\psdots[linecolor=black, dotsize=0.1](7.397115,-6.57)
\psdots[linecolor=black, dotsize=0.1](7.7971153,-6.57)
\psdots[linecolor=black, dotsize=0.4](8.997115,-5.37)
\psdots[linecolor=black, dotsize=0.4](9.797115,-6.57)
\psline[linecolor=black, linewidth=0.08](8.997115,-5.37)(8.197115,-6.57)(9.797115,-6.57)(8.997115,-5.37)(8.997115,-5.37)
\psdots[linecolor=black, dotsize=0.4](8.197115,-6.57)
\psdots[linecolor=black, dotsize=0.4](10.5971155,-5.37)
\psdots[linecolor=black, dotsize=0.4](11.397116,-6.57)
\psline[linecolor=black, linewidth=0.08](10.5971155,-5.37)(9.797115,-6.57)(11.397116,-6.57)(10.5971155,-5.37)(10.5971155,-5.37)
\psdots[linecolor=black, dotsize=0.4](9.797115,-6.57)
\psdots[linecolor=black, dotsize=0.1](11.797115,-6.57)
\psdots[linecolor=black, dotsize=0.1](12.197115,-6.57)
\psdots[linecolor=black, dotsize=0.1](12.5971155,-6.57)
\psdots[linecolor=black, dotsize=0.1](14.997115,-6.57)
\psdots[linecolor=black, dotsize=0.1](15.397116,-6.57)
\psdots[linecolor=black, dotsize=0.1](15.797115,-6.57)
\psdots[linecolor=black, dotsize=0.4](16.997116,-5.37)
\psdots[linecolor=black, dotsize=0.4](16.197115,-6.57)
\psdots[linecolor=black, dotsize=0.4](17.797115,-6.57)
\psdots[linecolor=black, dotsize=0.4](18.597115,-5.37)
\psdots[linecolor=black, dotsize=0.4](19.397116,-6.57)
\psline[linecolor=blue, linewidth=0.08](4.9971156,-6.57)(5.7971153,-5.37)(5.7971153,-5.37)
\psline[linecolor=blue, linewidth=0.08](14.5971155,-6.57)(13.797115,-5.37)(13.797115,-5.37)
\psline[linecolor=green, linewidth=0.08](5.7971153,-5.37)(6.5971155,-6.57)(6.5971155,-6.57)
\psline[linecolor=green, linewidth=0.08](13.797115,-5.37)(12.997115,-6.57)(12.997115,-6.57)
\psline[linecolor=red, linewidth=0.08](4.9971156,-6.57)(6.5971155,-6.57)(6.5971155,-6.57)
\psline[linecolor=red, linewidth=0.08](12.997115,-6.57)(14.5971155,-6.57)(14.5971155,-6.57)
\psdots[linecolor=black, dotsize=0.4](5.7971153,-5.37)
\psdots[linecolor=black, dotsize=0.4](4.9971156,-6.57)
\psdots[linecolor=black, dotsize=0.4](6.5971155,-6.57)
\psdots[linecolor=black, dotsize=0.4](12.997115,-6.57)
\psdots[linecolor=black, dotsize=0.4](13.797115,-5.37)
\psdots[linecolor=black, dotsize=0.4](14.5971155,-6.57)
\psline[linecolor=colour0, linewidth=0.08, linestyle=dashed, dash=0.17638889cm 0.10583334cm](9.797115,-4.57)(9.797115,-8.17)(9.797115,-8.17)
\rput[bl](5.6371155,-5.09){$u_i$}
\rput[bl](4.6771154,-7.13){$v_i$}
\rput[bl](6.5371156,-7.11){$w_i$}
\rput[bl](0.6971154,-7.95){(1)}
\rput[bl](2.4171154,-7.95){(2)}
\rput[bl](5.5571156,-7.97){(i)}
\rput[bl](8.717115,-8.01){(k)}
\rput[bl](10.197115,-8.03){(k+1)}
\rput[bl](16.597115,-8.05){(2k-1)}
\rput[bl](18.337116,-8.03){(2k)}
\psline[linecolor=black, linewidth=0.08](0.1971154,-6.57)(3.3971155,-6.57)(2.5971155,-5.37)(1.7971154,-6.57)(0.9971154,-5.37)(0.1971154,-6.57)(0.1971154,-6.57)
\psline[linecolor=black, linewidth=0.08](16.197115,-6.57)(16.997116,-5.37)(17.797115,-6.57)(18.597115,-5.37)(19.397116,-6.57)(16.197115,-6.57)(16.197115,-6.57)
\rput[bl](13.177115,-5.01){$u_{2k-i+1}$}
\rput[bl](14.197115,-7.09){$v_{2k-i+1}$}
\rput[bl](12.617115,-7.13){$w_{2k-i+1}$}
\rput[bl](13.257115,-8.05){(2k-i+1)}
\end{pspicture}
}
\end{center}
\caption{Chain triangular cactus $T_{2k}$ } \label{Chaintri2k}
\end{figure}

	\begin{proof}
		\begin{enumerate}
\item[(i)] 
We consider the following cases:

\begin{itemize}
\item[\textbf{Case 1.}] Suppose that $n$ is even, and  $n=2k$ for some $k\in \mathbb{N}$.
Consider the $T_{2k}$ as shown in Figure \ref{Chaintri2k}. One can easily check that whatever happens to computation of Mostar index related to the edge $u_iv_i$ in the $(i)$-th triangle in $T_{2k}$, is the same as computation of Mostar index related to the edge $u_{2k-i+1}v_{2k-i+1}$ in the $(2k-i+1)$-th triangle. The same goes for $w_iv_i$ and $w_{2k-i+1}v_{2k-i+1}$, and also for $w_iu_i$ and $w_{2k-i+1}u_{2k-i+1}$. So for computing Mostar index, it suffices to compute the $|n_u(uv,T_{2k})-n_v(uv,T_{2k})|$  for every $uv \in E(T_{2k})$ in the first $k$ triangles and then multiple that by 2. So from now, we only consider the  first $k$ triangles. 

Consider the blue edge $u_iv_i$ in the $(i)$-th triangle. There are $2(i-1)$ vertices which are closer to $v_i$ than $u_i$, but there are no vertices closer to $u_i$ than $v_i$. So, $|n_{u_i}(u_iv_i,T_{2k})-n_{v_i}(u_iv_i,T_{2k})|=2(i-1)$.

Now consider the green edge $u_iw_i$ in the $(i)$-th triangle. There are $2(2k-i)$ vertices which are closer to $w_i$ than $u_i$, but there are no vertices closer to $u_i$ than $w_i$. So, $|n_{u_i}(u_iw_i,T_{2k})-n_{w_i}(u_iw_i,T_{2k})|=2(2k-i)$.

Finally, consider the red edge $v_iw_i$ in the $(i)$-th triangle. There are $2(2k-i)$ vertices which are closer to $w_i$ than $v_i$, and there are $2(i-1)$ vertices closer to $v_i$ than $w_i$. So, $|n_{v_i}(v_iw_i,T_{2k})-n_{w_i}(v_iw_i,T_{2k})|=2(2k-2i+1)$.

Since we have $k$ edges like blue one, $k$ edges like green one and $k$ edges like red one, then by our argument, we have:

\begin{align*}
	Mo(T_{2k})&=2\left(\sum_{i=1}^{k}2(i-1) + \sum_{i=1}^{k}2(2k-i) + \sum_{i=1}^{k}2(2k-2i+1)\right)\\
	&=12k^2-4k.
\end{align*}

\item[\textbf{Case 2.}] Suppose that $n$ is odd and  $n=2k+1$ for some $k\in \mathbb{N}$.
Now consider the $T_{2k+1}$ as shown in Figure \ref{Chaintri2k+1}. One can easily check that whatever happens to computation of Mostar index related to the edge $u_iv_i$ in the $(i)$-th triangle in $T_{2k+1}$, is the same as computation of Mostar index related to the edge $u_{2k-i+2}v_{2k-i+2}$ in the $(2k-i+2)$-th triangle. The same goes for $w_iv_i$ and $w_{2k-i+2}v_{2k-i+2}$, and also for $w_iu_i$ and $w_{2k-i+2}u_{2k-i+2}$. So for computing Mostar index, it suffices to compute the $|n_u(uv,T_{2k+1})-n_v(uv,T_{2k+1})|$  for every $uv \in E(T_{2k+1})$ in the first $k$ triangles and then multiple that by 2 and add it to $\sum_{uv\in A}^{}|n_u(uv,T_{2k+1})-n_v(uv,T_{2k+1})|$, where  $A = \{ab,bc,ac\}$. So from now, we only consider the  first $k$ triangles and the middle one. 

Consider the blue edge $u_iv_i$ in the $(i)$-th triangle. There are $2(i-1)$ vertices which are closer to $v_i$ than $u_i$, but there are no vertices closer to $u_i$ than $v_i$. So, $|n_{u_i}(u_iv_i,T_{2k+1})-n_{v_i}(u_iv_i,T_{2k+1})|=2(i-1)$.

Now consider the green edge $u_iw_i$ in the $(i)$-th triangle. There are $4k-2i+2$ vertices which are closer to $w_i$ than $u_i$, but there are no vertices closer to $u_i$ than $w_i$. So, $|n_{u_i}(u_iw_i,T_{2k+1})-n_{w_i}(u_iw_i,T_{2k+1})|=2(2k-i+1)$.

Now consider the red edge $v_iw_i$ in the $(i)$-th triangle. There are $2(2k-i+1)$ vertices which are closer to $w_i$ than $v_i$, and there are $2(i-1)$ vertices closer to $v_i$ than $w_i$. So, $|n_{v_i}(v_iw_i,T_{2k+1})-n_{w_i}(v_iw_i,T_{2k+1})|=4(k-i+1)$.

Finally, consider  the middle triangle. For the  edge $ab$, there are $2k$ vertices which are closer to $b$ than $a$, but there are no vertices closer to $a$ than $b$. Also for the  edge $ac$, there are $2k$ vertices which are closer to $c$ than $a$, but there are no vertices closer to $a$ than $c$ and for the  edge $bc$, there are $2k$ vertices which are closer to $b$ than $c$, and  there are $2k$ vertices closer to $c$ than $b$. Hence, $\sum_{uv\in A}^{}|n_u(uv,T_{2k+1})-n_v(uv,T_{2k+1})|=4k$, where  $A = \{ab,bc,ac\}$.

Since we have $k$ edges like blue one, $k$ edges like green one and $k$ edges like red one, then by our argument, we have:

\begin{align*}
	Mo(T_{2k+1})&=2\left(\sum_{i=1}^{k}2(i-1) + \sum_{i=1}^{k}2(2k-i+1) + \sum_{i=1}^{k}4(k-i+1)\right)+4k\\
	&=12k^2+8k.
\end{align*}

\end{itemize}	
Therefore, we have the result. 

\item[(ii)] The proof is similar to proof of Part (i). \qed
\end{enumerate} 
	\end{proof}

\begin{figure}
\begin{center}
\psscalebox{0.7 0.7}
{
\begin{pspicture}(0,-8.22)(21.19423,-4.56)
\definecolor{colour0}{rgb}{1.0,0.6,0.4}
\psdots[linecolor=black, dotsize=0.4](0.9971154,-5.37)
\psdots[linecolor=black, dotsize=0.4](0.19711538,-6.57)
\psdots[linecolor=black, dotsize=0.4](1.7971153,-6.57)
\psdots[linecolor=black, dotsize=0.4](2.5971153,-5.37)
\psdots[linecolor=black, dotsize=0.4](3.3971155,-6.57)
\psdots[linecolor=black, dotsize=0.1](3.7971153,-6.57)
\psdots[linecolor=black, dotsize=0.1](4.1971154,-6.57)
\psdots[linecolor=black, dotsize=0.1](4.5971155,-6.57)
\psdots[linecolor=black, dotsize=0.1](6.9971156,-6.57)
\psdots[linecolor=black, dotsize=0.1](7.397115,-6.57)
\psdots[linecolor=black, dotsize=0.1](7.7971153,-6.57)
\psdots[linecolor=black, dotsize=0.4](8.997115,-5.37)
\psdots[linecolor=black, dotsize=0.4](9.797115,-6.57)
\psline[linecolor=black, linewidth=0.08](8.997115,-5.37)(8.197115,-6.57)(9.797115,-6.57)(8.997115,-5.37)(8.997115,-5.37)
\psdots[linecolor=black, dotsize=0.4](8.197115,-6.57)
\psdots[linecolor=black, dotsize=0.4](10.5971155,-5.37)
\psdots[linecolor=black, dotsize=0.4](11.397116,-6.57)
\psline[linecolor=black, linewidth=0.08](10.5971155,-5.37)(9.797115,-6.57)(11.397116,-6.57)(10.5971155,-5.37)(10.5971155,-5.37)
\psdots[linecolor=black, dotsize=0.4](9.797115,-6.57)
\psdots[linecolor=black, dotsize=0.1](13.397116,-6.57)
\psdots[linecolor=black, dotsize=0.1](13.797115,-6.57)
\psdots[linecolor=black, dotsize=0.1](14.197115,-6.57)
\psdots[linecolor=black, dotsize=0.1](16.597115,-6.57)
\psdots[linecolor=black, dotsize=0.1](16.997116,-6.57)
\psdots[linecolor=black, dotsize=0.1](17.397116,-6.57)
\psdots[linecolor=black, dotsize=0.4](18.597115,-5.37)
\psdots[linecolor=black, dotsize=0.4](17.797115,-6.57)
\psdots[linecolor=black, dotsize=0.4](19.397116,-6.57)
\psdots[linecolor=black, dotsize=0.4](20.197115,-5.37)
\psdots[linecolor=black, dotsize=0.4](20.997116,-6.57)
\psline[linecolor=blue, linewidth=0.08](4.9971156,-6.57)(5.7971153,-5.37)(5.7971153,-5.37)
\psline[linecolor=blue, linewidth=0.08](16.197115,-6.57)(15.397116,-5.37)(15.397116,-5.37)
\psline[linecolor=green, linewidth=0.08](5.7971153,-5.37)(6.5971155,-6.57)(6.5971155,-6.57)
\psline[linecolor=green, linewidth=0.08](15.397116,-5.37)(14.5971155,-6.57)(14.5971155,-6.57)
\psline[linecolor=red, linewidth=0.08](4.9971156,-6.57)(6.5971155,-6.57)(6.5971155,-6.57)
\psline[linecolor=red, linewidth=0.08](14.5971155,-6.57)(16.197115,-6.57)(16.197115,-6.57)
\psdots[linecolor=black, dotsize=0.4](5.7971153,-5.37)
\psdots[linecolor=black, dotsize=0.4](4.9971156,-6.57)
\psdots[linecolor=black, dotsize=0.4](6.5971155,-6.57)
\psdots[linecolor=black, dotsize=0.4](14.5971155,-6.57)
\psdots[linecolor=black, dotsize=0.4](15.397116,-5.37)
\psdots[linecolor=black, dotsize=0.4](16.197115,-6.57)
\rput[bl](5.6371155,-5.09){$u_i$}
\rput[bl](4.6771154,-7.13){$v_i$}
\rput[bl](6.5371156,-7.11){$w_i$}
\rput[bl](0.69711536,-7.95){(1)}
\rput[bl](2.4171154,-7.95){(2)}
\rput[bl](5.5571156,-7.97){(i)}
\rput[bl](8.717115,-8.01){(k)}
\rput[bl](10.197115,-8.03){(k+1)}
\rput[bl](18.337116,-8.03){(2k)}
\psline[linecolor=black, linewidth=0.08](0.19711538,-6.57)(3.3971155,-6.57)(2.5971153,-5.37)(1.7971153,-6.57)(0.9971154,-5.37)(0.19711538,-6.57)(0.19711538,-6.57)
\psline[linecolor=black, linewidth=0.08](17.797115,-6.57)(18.597115,-5.37)(19.397116,-6.57)(20.197115,-5.37)(20.997116,-6.57)(17.797115,-6.57)(17.797115,-6.57)
\psline[linecolor=black, linewidth=0.08](11.397116,-6.57)(12.197115,-5.37)(12.997115,-6.57)(11.397116,-6.57)(11.397116,-6.57)
\psdots[linecolor=black, dotsize=0.4](12.197115,-5.37)
\psdots[linecolor=black, dotsize=0.4](12.997115,-6.57)
\psline[linecolor=colour0, linewidth=0.08, linestyle=dashed, dash=0.17638889cm 0.10583334cm](11.397116,-4.57)(11.397116,-8.17)(11.397116,-8.17)
\psline[linecolor=colour0, linewidth=0.08, linestyle=dashed, dash=0.17638889cm 0.10583334cm](9.797115,-4.57)(9.797115,-8.17)(9.797115,-8.17)
\rput[bl](11.6371155,-8.05){$(k+2)$}
\rput[bl](14.497115,-8.01){$(2k-i+2)$}
\rput[bl](19.637115,-8.03){$(2k+1)$}
\rput[bl](14.717115,-5.09){$u_{2k-i+2}$}
\rput[bl](15.777116,-7.19){$v_{2k-i+2}$}
\rput[bl](13.877115,-7.17){$w_{2k-i+2}$}
\rput[bl](10.437116,-4.97){a}
\rput[bl](9.877115,-7.11){b}
\rput[bl](11.057116,-7.07){c}
\end{pspicture}
}
\end{center}
\caption{Chain triangular cactus $T_{2k+1}$ } \label{Chaintri2k+1}
\end{figure}
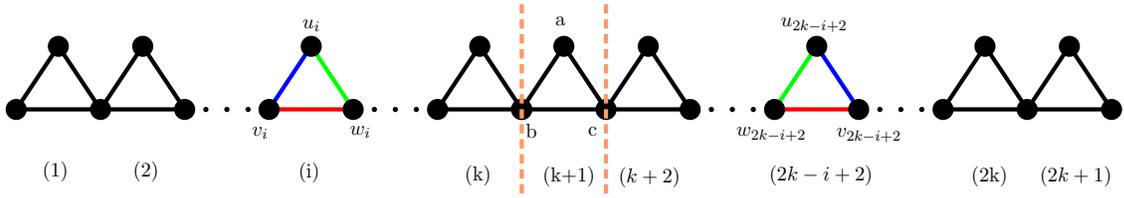

	\begin{theorem} \label{thm-para-Q}
Let $Q_n$ be the para-chain square cactus graph  of order $n$. Then for every $n\geq 1$, and $k \geq 1$, we have:
\begin{enumerate} 
	\item[(i)] 
	 \[
 	Mo(Q_n)=\left\{
  	\begin{array}{ll}
  	{\displaystyle
  		24k^2}&
  		\quad\mbox{if $n=2k$, }\\[15pt]
  		{\displaystyle
  			24k^2+24k}&
  			\quad\mbox{if $n=2k+1$,}
  				  \end{array}
  					\right.	
  					\]
\item[(ii)] 
\[
Mo_e(Q_n)=\left\{
\begin{array}{ll}
{\displaystyle
	32k^2}&
\quad\mbox{if $n=2k$, }\\[15pt]
{\displaystyle
	32k^2+32k}&
\quad\mbox{if $n=2k+1$,}
\end{array}
\right.	
\]
\end{enumerate}
	\end{theorem}

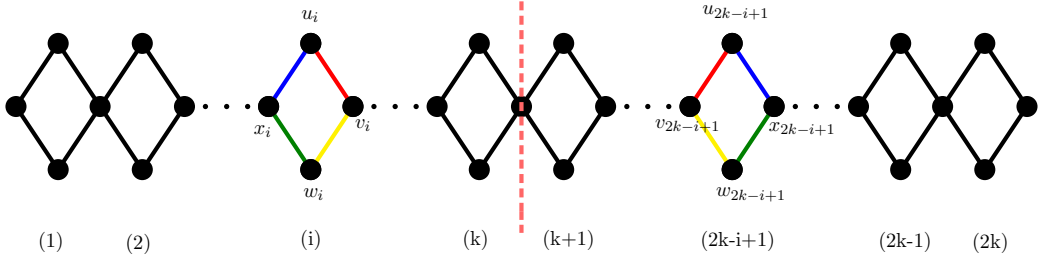
\begin{figure}
\begin{center}
\psscalebox{0.7 0.7}
{
\begin{pspicture}(0,-5.195)(19.59423,-0.385)
\definecolor{colour1}{rgb}{0.0,0.5019608,0.0}
\definecolor{colour0}{rgb}{1.0,0.4,0.4}
\psdots[linecolor=black, dotsize=0.4](0.9971153,-1.195)
\psdots[linecolor=black, dotsize=0.4](0.19711533,-2.395)
\psdots[linecolor=black, dotsize=0.4](0.9971153,-3.595)
\psdots[linecolor=black, dotsize=0.4](1.7971153,-2.395)
\psdots[linecolor=black, dotsize=0.4](2.5971153,-1.195)
\psdots[linecolor=black, dotsize=0.4](3.3971152,-2.395)
\psdots[linecolor=black, dotsize=0.4](2.5971153,-3.595)
\psdots[linecolor=black, dotsize=0.1](3.7971153,-2.395)
\psdots[linecolor=black, dotsize=0.1](4.1971154,-2.395)
\psdots[linecolor=black, dotsize=0.1](4.5971155,-2.395)
\psdots[linecolor=black, dotsize=0.4](4.997115,-2.395)
\psdots[linecolor=black, dotsize=0.4](5.7971153,-1.195)
\psdots[linecolor=black, dotsize=0.4](6.5971155,-2.395)
\psdots[linecolor=black, dotsize=0.4](5.7971153,-3.595)
\psdots[linecolor=black, dotsize=0.1](6.997115,-2.395)
\psdots[linecolor=black, dotsize=0.1](7.397115,-2.395)
\psdots[linecolor=black, dotsize=0.1](7.7971153,-2.395)
\psdots[linecolor=black, dotsize=0.4](8.997115,-1.195)
\psdots[linecolor=black, dotsize=0.4](8.197115,-2.395)
\psdots[linecolor=black, dotsize=0.4](8.997115,-3.595)
\psdots[linecolor=black, dotsize=0.4](9.797115,-2.395)
\psdots[linecolor=black, dotsize=0.4](10.5971155,-1.195)
\psdots[linecolor=black, dotsize=0.4](11.397116,-2.395)
\psdots[linecolor=black, dotsize=0.4](10.5971155,-3.595)
\psdots[linecolor=black, dotsize=0.1](11.797115,-2.395)
\psdots[linecolor=black, dotsize=0.1](12.197115,-2.395)
\psdots[linecolor=black, dotsize=0.1](12.5971155,-2.395)
\psdots[linecolor=black, dotsize=0.4](12.997115,-2.395)
\psdots[linecolor=black, dotsize=0.4](13.797115,-1.195)
\psdots[linecolor=black, dotsize=0.4](14.5971155,-2.395)
\psdots[linecolor=black, dotsize=0.4](13.797115,-3.595)
\psdots[linecolor=black, dotsize=0.1](14.997115,-2.395)
\psdots[linecolor=black, dotsize=0.1](15.397116,-2.395)
\psdots[linecolor=black, dotsize=0.1](15.797115,-2.395)
\psdots[linecolor=black, dotsize=0.4](16.997116,-1.195)
\psdots[linecolor=black, dotsize=0.4](16.197115,-2.395)
\psdots[linecolor=black, dotsize=0.4](16.997116,-3.595)
\psdots[linecolor=black, dotsize=0.4](17.797115,-2.395)
\psdots[linecolor=black, dotsize=0.4](18.597115,-1.195)
\psdots[linecolor=black, dotsize=0.4](19.397116,-2.395)
\psdots[linecolor=black, dotsize=0.4](18.597115,-3.595)
\psline[linecolor=black, linewidth=0.08](0.9971153,-1.195)(0.19711533,-2.395)(0.9971153,-3.595)(2.5971153,-1.195)(3.3971152,-2.395)(2.5971153,-3.595)(0.9971153,-1.195)(0.9971153,-1.195)
\psline[linecolor=black, linewidth=0.08](8.197115,-2.395)(8.997115,-1.195)(10.5971155,-3.595)(11.397116,-2.395)(10.5971155,-1.195)(8.997115,-3.595)(8.197115,-2.395)(8.197115,-2.395)
\psline[linecolor=black, linewidth=0.08](16.197115,-2.395)(16.997116,-1.195)(18.597115,-3.595)(19.397116,-2.395)(18.597115,-1.195)(16.997116,-3.595)(16.197115,-2.395)(16.197115,-2.395)
\psline[linecolor=blue, linewidth=0.08](4.997115,-2.395)(5.7971153,-1.195)(5.7971153,-1.195)
\psline[linecolor=blue, linewidth=0.08](13.797115,-1.195)(14.5971155,-2.395)(14.5971155,-2.395)
\psline[linecolor=red, linewidth=0.08](5.7971153,-1.195)(6.5971155,-2.395)(6.5971155,-2.395)
\psline[linecolor=red, linewidth=0.08](13.797115,-1.195)(12.997115,-2.395)(12.997115,-2.395)
\psline[linecolor=colour1, linewidth=0.08](4.997115,-2.395)(5.7971153,-3.595)(5.7971153,-3.595)
\psline[linecolor=colour1, linewidth=0.08](14.5971155,-2.395)(13.797115,-3.595)(13.797115,-3.595)
\psline[linecolor=yellow, linewidth=0.08](12.997115,-2.395)(13.797115,-3.595)(13.797115,-3.595)
\psline[linecolor=yellow, linewidth=0.08](6.5971155,-2.395)(5.7971153,-3.595)(5.7971153,-3.595)
\psdots[linecolor=black, dotsize=0.1](4.997115,-2.395)
\psdots[linecolor=black, dotsize=0.1](4.997115,-2.395)
\psdots[linecolor=black, dotsize=0.1](4.997115,-2.395)
\psdots[linecolor=black, dotsize=0.1](4.997115,-2.395)
\psdots[linecolor=black, dotsize=0.4](4.997115,-2.395)
\psdots[linecolor=black, dotsize=0.4](5.7971153,-1.195)
\psdots[linecolor=black, dotsize=0.4](6.5971155,-2.395)
\psdots[linecolor=black, dotsize=0.4](5.7971153,-3.595)
\psdots[linecolor=black, dotsize=0.4](12.997115,-2.395)
\psdots[linecolor=black, dotsize=0.4](13.797115,-1.195)
\psdots[linecolor=black, dotsize=0.4](14.5971155,-2.395)
\psdots[linecolor=black, dotsize=0.4](13.797115,-3.595)
\psline[linecolor=colour0, linewidth=0.08, linestyle=dashed, dash=0.17638889cm 0.10583334cm](9.797115,-0.395)(9.797115,-4.395)(9.797115,-4.395)
\psline[linecolor=colour0, linewidth=0.08, linestyle=dashed, dash=0.17638889cm 0.10583334cm](9.797115,-4.395)(9.797115,-4.795)(9.797115,-4.795)
\rput[bl](0.6171153,-5.155){(1)}
\rput[bl](2.2771153,-5.195){(2)}
\rput[bl](5.5971155,-5.115){(i)}
\rput[bl](8.677115,-5.115){(k)}
\rput[bl](10.197115,-5.115){(k+1)}
\rput[bl](16.577116,-5.155){(2k-1)}
\rput[bl](18.357115,-5.175){(2k)}
\rput[bl](13.197115,-5.135){(2k-i+1)}
\rput[bl](5.6171155,-0.795){$u_i$}
\rput[bl](6.6371155,-2.855){$v_i$}
\rput[bl](5.6571155,-4.175){$w_i$}
\rput[bl](4.7171154,-2.975){$x_i$}
\rput[bl](13.237115,-0.735){$u_{2k-i+1}$}
\rput[bl](12.337115,-2.895){$v_{2k-i+1}$}
\rput[bl](13.497115,-4.195){$w_{2k-i+1}$}
\rput[bl](14.497115,-2.955){$x_{2k-i+1}$}
\end{pspicture}
}
\end{center}
\caption{Para-chain square cactus $Q_{2k}$ } \label{paraChainsqu2k}
\end{figure}

	\begin{proof}
		\begin{enumerate} 
			\item[(i)]
			We consider the following cases:

\begin{itemize}
\item[\textbf{Case 1.}] Suppose that $n$ is even and  $n=2k$ for some $k\in \mathbb{N}$.
Now consider the $Q_{2k}$ as shown in Figure \ref{paraChainsqu2k}. One can easily check that whatever happens to computation of Mostar index related to the edge $u_iv_i$ in the $(i)$-th rhombus in $Q_{2k}$, is the same as computation of Mostar index related to the edge $u_{2k-i+1}v_{2k-i+1}$ in the $(2k-i+1)$-th rhombus. The same goes for $w_iv_i$ and $w_{2k-i+1}v_{2k-i+1}$, for $w_ix_i$ and $w_{2k-i+1}x_{2k-i+1}$, and also for $x_iu_i$ and $x_{2k-i+1}u_{2k-i+1}$. So for computing Mostar index, it suffices to compute the $|n_u(uv,Q_{2k})-n_v(uv,Q_{2k})|$  for every $uv \in E(Q_{2k})$ in the first $k$ rhombus and then multiple that by 2. So from now, we only consider the  first $k$ rhombus.

Consider the red edge $u_iv_i$ in the $(i)$-th rhombus. There are $3k+3(k-i)+1$ vertices which are closer to $v_i$ than $u_i$, and there are $3i-2$ vertices closer to $u_i$ than $v_i$. So, $|n_{u_i}(u_iv_i,Q_{2k})-n_{v_i}(u_iv_i,Q_{2k})|=6k-6i+3$.

One can easily check that the edges $w_iv_i$, $w_ix_i$ and $x_iu_i$ have the same attitude as $u_iv_i$. Since we have $k$ edges like blue one, $k$ edges like green one, $k$ edges like yellow one and $k$ edges like red one, then by our argument, we have:

\begin{align*}
	Mo(Q_{2k})=2\left(4\sum_{i=1}^{k}3(2k-2i+1)\right)=24k^2.
\end{align*}

\item[\textbf{Case 2.}] Suppose that $n$ is odd and $n=2k+1$ for some $k\in \mathbb{N}$.
Now consider the $Q_{2k+1}$ as shown in Figure \ref{paraChainsqu2k+1}. One can easily check that whatever happens to computation of Mostar index related to the edge $u_iv_i$ in the $(i)$-th rhombus in $Q_{2k+1}$, is the same as computation of Mostar index related to the edge $u_{2k-i+2}v_{2k-i+2}$ in the $(2k-i+2)$-th rhombus. The same goes for $w_iv_i$ and $w_{2k-i+2}v_{2k-i+2}$, for $w_ix_i$ and $w_{2k-i+2}x_{2k-i+2}$, and also for $x_iu_i$ and $x_{2k-i+2}u_{2k-i+2}$. So for computing Mostar index, it suffices to compute the $|n_u(uv,Q_{2k+1})-n_v(uv,Q_{2k+1})|$  for every $uv \in E(Q_{2k+1})$ in the first $k$ rhombus and then multiple that by 2 and add it to $\sum_{uv\in A}^{}|n_u(uv,Q_{2k+1})-n_v(uv,Q_{2k+1})|$, where  $A = \{ab,bc,cd,da\}$. So from now, we only consider the  first $k+1$ rhombus.

Consider the red edge $u_iv_i$ in the $(i)$-th rhombus. There are $3(k+1)+3(k-i)+1$ vertices which are closer to $v_i$ than $u_i$, and there are $3i-2$ vertices closer to $u_i$ than $v_i$. So, $|n_{u_i}(u_iv_i,Q_{2k+1})-n_{v_i}(u_iv_i,Q_{2k+1})|=6k-6i+6$.

One can easily check that the edges $w_iv_i$, $w_ix_i$ and $x_iu_i$ have the same attitude as $u_iv_i$. 

Now consider the middle rhombus. For the  edge $ab$, there are $3k+1$ vertices which are closer to $b$ than $a$, and there are $3k+1$ vertices closer to $a$ than $b$. the edges $bc$, $cd$ and $da$ have the same attitude as $ab$. Hence, $\sum_{uv\in A}^{}|n_u(uv,Q_{2k+1})-n_v(uv,Q_{2k+1})|=0$, where  $A = \{ab,bc,cd,da\}$.

Since we have $k$ edges like blue one, $k$ edges like green one, $k$ edges like yellow one and $k$ edges like red one, then by our argument, we have:

\begin{align*}
	Mo(Q_{2k+1})=2\left(4\sum_{i=1}^{k}6(k-i+1)\right)=24k^2+24k.
\end{align*}

\end{itemize}	
Therefore, we have the result.

\item[(ii)] The proof is similar to the proof of Part (i). 
\qed
\end{enumerate} 
	\end{proof}

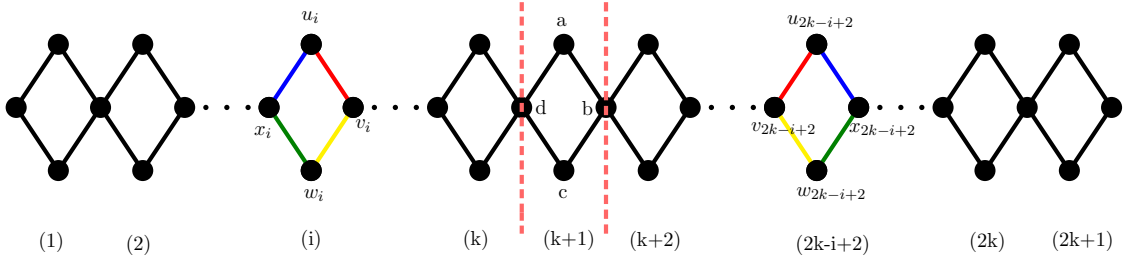
\begin{figure}
\begin{center}
\psscalebox{0.7 0.7}
{
\begin{pspicture}(0,-5.195)(21.19423,-0.385)
\definecolor{colour0}{rgb}{0.0,0.5019608,0.0}
\definecolor{colour1}{rgb}{1.0,0.4,0.4}
\psdots[linecolor=black, dotsize=0.4](0.9971154,-1.195)
\psdots[linecolor=black, dotsize=0.4](0.19711538,-2.395)
\psdots[linecolor=black, dotsize=0.4](0.9971154,-3.595)
\psdots[linecolor=black, dotsize=0.4](1.7971153,-2.395)
\psdots[linecolor=black, dotsize=0.4](2.5971153,-1.195)
\psdots[linecolor=black, dotsize=0.4](3.3971155,-2.395)
\psdots[linecolor=black, dotsize=0.4](2.5971153,-3.595)
\psdots[linecolor=black, dotsize=0.1](3.7971153,-2.395)
\psdots[linecolor=black, dotsize=0.1](4.1971154,-2.395)
\psdots[linecolor=black, dotsize=0.1](4.5971155,-2.395)
\psdots[linecolor=black, dotsize=0.4](4.9971156,-2.395)
\psdots[linecolor=black, dotsize=0.4](5.7971153,-1.195)
\psdots[linecolor=black, dotsize=0.4](6.5971155,-2.395)
\psdots[linecolor=black, dotsize=0.4](5.7971153,-3.595)
\psdots[linecolor=black, dotsize=0.1](6.9971156,-2.395)
\psdots[linecolor=black, dotsize=0.1](7.397115,-2.395)
\psdots[linecolor=black, dotsize=0.1](7.7971153,-2.395)
\psdots[linecolor=black, dotsize=0.4](8.997115,-1.195)
\psdots[linecolor=black, dotsize=0.4](8.197115,-2.395)
\psdots[linecolor=black, dotsize=0.4](8.997115,-3.595)
\psdots[linecolor=black, dotsize=0.4](9.797115,-2.395)
\psdots[linecolor=black, dotsize=0.4](10.5971155,-1.195)
\psdots[linecolor=black, dotsize=0.4](11.397116,-2.395)
\psdots[linecolor=black, dotsize=0.4](10.5971155,-3.595)
\psdots[linecolor=black, dotsize=0.1](13.397116,-2.395)
\psdots[linecolor=black, dotsize=0.1](13.797115,-2.395)
\psdots[linecolor=black, dotsize=0.1](14.197115,-2.395)
\psdots[linecolor=black, dotsize=0.4](14.5971155,-2.395)
\psdots[linecolor=black, dotsize=0.4](15.397116,-1.195)
\psdots[linecolor=black, dotsize=0.4](16.197115,-2.395)
\psdots[linecolor=black, dotsize=0.4](15.397116,-3.595)
\psdots[linecolor=black, dotsize=0.1](16.597115,-2.395)
\psdots[linecolor=black, dotsize=0.1](16.997116,-2.395)
\psdots[linecolor=black, dotsize=0.1](17.397116,-2.395)
\psdots[linecolor=black, dotsize=0.4](18.597115,-1.195)
\psdots[linecolor=black, dotsize=0.4](17.797115,-2.395)
\psdots[linecolor=black, dotsize=0.4](18.597115,-3.595)
\psdots[linecolor=black, dotsize=0.4](19.397116,-2.395)
\psdots[linecolor=black, dotsize=0.4](20.197115,-1.195)
\psdots[linecolor=black, dotsize=0.4](20.997116,-2.395)
\psdots[linecolor=black, dotsize=0.4](20.197115,-3.595)
\psline[linecolor=black, linewidth=0.08](0.9971154,-1.195)(0.19711538,-2.395)(0.9971154,-3.595)(2.5971153,-1.195)(3.3971155,-2.395)(2.5971153,-3.595)(0.9971154,-1.195)(0.9971154,-1.195)
\psline[linecolor=black, linewidth=0.08](8.197115,-2.395)(8.997115,-1.195)(10.5971155,-3.595)(11.397116,-2.395)(10.5971155,-1.195)(8.997115,-3.595)(8.197115,-2.395)(8.197115,-2.395)
\psline[linecolor=black, linewidth=0.08](17.797115,-2.395)(18.597115,-1.195)(20.197115,-3.595)(20.997116,-2.395)(20.197115,-1.195)(18.597115,-3.595)(17.797115,-2.395)(17.797115,-2.395)
\psline[linecolor=blue, linewidth=0.08](4.9971156,-2.395)(5.7971153,-1.195)(5.7971153,-1.195)
\psline[linecolor=blue, linewidth=0.08](15.397116,-1.195)(16.197115,-2.395)(16.197115,-2.395)
\psline[linecolor=red, linewidth=0.08](5.7971153,-1.195)(6.5971155,-2.395)(6.5971155,-2.395)
\psline[linecolor=red, linewidth=0.08](15.397116,-1.195)(14.5971155,-2.395)(14.5971155,-2.395)
\psline[linecolor=colour0, linewidth=0.08](4.9971156,-2.395)(5.7971153,-3.595)(5.7971153,-3.595)
\psline[linecolor=colour0, linewidth=0.08](16.197115,-2.395)(15.397116,-3.595)(15.397116,-3.595)
\psline[linecolor=yellow, linewidth=0.08](14.5971155,-2.395)(15.397116,-3.595)(15.397116,-3.595)
\psline[linecolor=yellow, linewidth=0.08](6.5971155,-2.395)(5.7971153,-3.595)(5.7971153,-3.595)
\psdots[linecolor=black, dotsize=0.1](4.9971156,-2.395)
\psdots[linecolor=black, dotsize=0.1](4.9971156,-2.395)
\psdots[linecolor=black, dotsize=0.1](4.9971156,-2.395)
\psdots[linecolor=black, dotsize=0.1](4.9971156,-2.395)
\psdots[linecolor=black, dotsize=0.4](4.9971156,-2.395)
\psdots[linecolor=black, dotsize=0.4](5.7971153,-1.195)
\psdots[linecolor=black, dotsize=0.4](6.5971155,-2.395)
\psdots[linecolor=black, dotsize=0.4](5.7971153,-3.595)
\psdots[linecolor=black, dotsize=0.4](14.5971155,-2.395)
\psdots[linecolor=black, dotsize=0.4](15.397116,-1.195)
\psdots[linecolor=black, dotsize=0.4](16.197115,-2.395)
\psdots[linecolor=black, dotsize=0.4](15.397116,-3.595)
\psline[linecolor=colour1, linewidth=0.08, linestyle=dashed, dash=0.17638889cm 0.10583334cm](9.797115,-0.395)(9.797115,-4.395)(9.797115,-4.395)
\psline[linecolor=colour1, linewidth=0.08, linestyle=dashed, dash=0.17638889cm 0.10583334cm](9.797115,-4.395)(9.797115,-4.795)(9.797115,-4.795)
\rput[bl](0.6171154,-5.155){(1)}
\rput[bl](2.2771153,-5.195){(2)}
\rput[bl](5.5971155,-5.115){(i)}
\rput[bl](8.677115,-5.115){(k)}
\rput[bl](10.197115,-5.115){(k+1)}
\rput[bl](18.297115,-5.155){(2k)}
\rput[bl](5.6171155,-0.795){$u_i$}
\rput[bl](6.6371155,-2.855){$v_i$}
\rput[bl](5.6571155,-4.175){$w_i$}
\rput[bl](4.7171154,-2.975){$x_i$}
\psdots[linecolor=black, dotsize=0.4](12.197115,-1.195)
\psdots[linecolor=black, dotsize=0.4](12.997115,-2.395)
\psdots[linecolor=black, dotsize=0.4](12.197115,-3.595)
\psline[linecolor=black, linewidth=0.08](11.397116,-2.395)(12.197115,-1.195)(12.997115,-2.395)(12.197115,-3.595)(11.397116,-2.395)(11.397116,-2.395)
\psline[linecolor=colour1, linewidth=0.08, linestyle=dashed, dash=0.17638889cm 0.10583334cm](11.397116,-0.395)(11.397116,-4.795)(11.397116,-4.795)
\rput[bl](14.997115,-5.195){(2k-i+2)}
\rput[bl](14.837115,-0.895){$u_{2k-i+2}$}
\rput[bl](14.157115,-2.935){$v_{2k-i+2}$}
\rput[bl](15.017116,-4.175){$w_{2k-i+2}$}
\rput[bl](16.017115,-2.955){$x_{2k-i+2}$}
\rput[bl](11.837115,-5.135){(k+2)}
\rput[bl](10.457115,-0.855){a}
\rput[bl](10.937116,-2.535){b}
\rput[bl](10.497115,-4.115){c}
\rput[bl](10.057116,-2.515){d}
\rput[bl](19.857115,-5.135){(2k+1)}
\end{pspicture}
}
\end{center}
\caption{Para-chain square cactus $Q_{2k+1}$ } \label{paraChainsqu2k+1}
\end{figure}

	\begin{theorem} \label{thm-para-O}
Let $O_n$ be the para-chain square cactus graph  of order $n$. Then for every $n\geq 1$, and $k \geq 1$, we have:
\begin{enumerate}
	\item [(i)]
 \[
 	Mo(O_n)=\left\{
  	\begin{array}{ll}
  	{\displaystyle
  		36k^2-12k}&
  		\quad\mbox{if $n=2k$, }\\[15pt]
  		{\displaystyle
  			36k^2+24k}&
  			\quad\mbox{if $n=2k+1$.}
  				  \end{array}
  					\right.	
  					\]
\item[(ii)] 
\[
Mo_e(O_n)=\left\{
\begin{array}{ll}
{\displaystyle
	48k^2-16k}&
\quad\mbox{if $n=2k$, }\\[15pt]
{\displaystyle
	48k^2+32k}&
\quad\mbox{if $n=2k+1$.}
\end{array}
\right.	
\]
\end{enumerate} 
	\end{theorem}

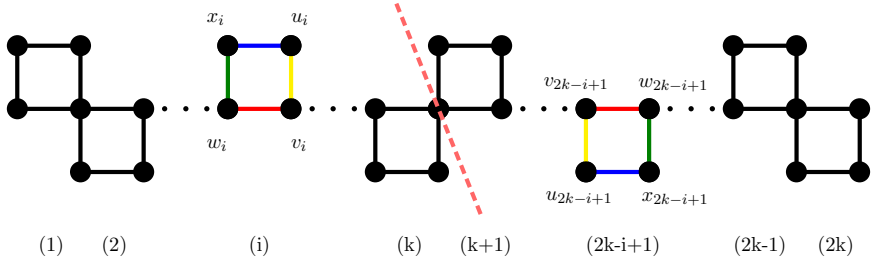
\begin{figure}
\begin{center}
\psscalebox{0.7 0.7}
{
\begin{pspicture}(0,-5.2176895)(16.394232,-0.33259934)
\definecolor{colour0}{rgb}{0.0,0.5019608,0.0}
\definecolor{colour1}{rgb}{1.0,0.4,0.4}
\psdots[linecolor=black, dotsize=0.4](0.1971154,-1.157455)
\psdots[linecolor=black, dotsize=0.4](1.3971153,-1.157455)
\psdots[linecolor=black, dotsize=0.4](0.1971154,-2.357455)
\psdots[linecolor=black, dotsize=0.4](1.3971153,-2.357455)
\psdots[linecolor=black, dotsize=0.4](2.5971155,-2.357455)
\psdots[linecolor=black, dotsize=0.4](1.3971153,-3.557455)
\psdots[linecolor=black, dotsize=0.4](2.5971155,-3.557455)
\psdots[linecolor=black, dotsize=0.1](2.9971154,-2.357455)
\psdots[linecolor=black, dotsize=0.1](3.3971155,-2.357455)
\psdots[linecolor=black, dotsize=0.1](3.7971153,-2.357455)
\psdots[linecolor=black, dotsize=0.4](4.1971154,-2.357455)
\psdots[linecolor=black, dotsize=0.4](4.1971154,-1.157455)
\psdots[linecolor=black, dotsize=0.4](5.397115,-1.157455)
\psdots[linecolor=black, dotsize=0.4](5.397115,-2.357455)
\psdots[linecolor=black, dotsize=0.1](5.7971153,-2.357455)
\psdots[linecolor=black, dotsize=0.1](6.1971154,-2.357455)
\psdots[linecolor=black, dotsize=0.1](6.5971155,-2.357455)
\psdots[linecolor=black, dotsize=0.4](6.9971156,-2.357455)
\psdots[linecolor=black, dotsize=0.4](8.197115,-2.357455)
\psdots[linecolor=black, dotsize=0.4](6.9971156,-3.557455)
\psdots[linecolor=black, dotsize=0.4](8.197115,-3.557455)
\psdots[linecolor=black, dotsize=0.4](8.197115,-1.157455)
\psdots[linecolor=black, dotsize=0.4](9.397116,-1.157455)
\psdots[linecolor=black, dotsize=0.4](8.197115,-2.357455)
\psdots[linecolor=black, dotsize=0.4](9.397116,-2.357455)
\psdots[linecolor=black, dotsize=0.1](9.797115,-2.357455)
\psdots[linecolor=black, dotsize=0.1](10.197115,-2.357455)
\psdots[linecolor=black, dotsize=0.1](10.5971155,-2.357455)
\psdots[linecolor=black, dotsize=0.4](10.997115,-2.357455)
\psdots[linecolor=black, dotsize=0.4](10.997115,-3.557455)
\psdots[linecolor=black, dotsize=0.4](12.197115,-3.557455)
\psdots[linecolor=black, dotsize=0.4](12.197115,-2.357455)
\psdots[linecolor=black, dotsize=0.1](12.5971155,-2.357455)
\psdots[linecolor=black, dotsize=0.1](12.997115,-2.357455)
\psdots[linecolor=black, dotsize=0.1](13.397116,-2.357455)
\psdots[linecolor=black, dotsize=0.4](13.797115,-2.357455)
\psdots[linecolor=black, dotsize=0.4](14.997115,-2.357455)
\psdots[linecolor=black, dotsize=0.4](13.797115,-1.157455)
\psdots[linecolor=black, dotsize=0.4](14.997115,-1.157455)
\psdots[linecolor=black, dotsize=0.4](14.997115,-3.557455)
\psdots[linecolor=black, dotsize=0.4](16.197115,-3.557455)
\psdots[linecolor=black, dotsize=0.4](16.197115,-2.357455)
\psline[linecolor=black, linewidth=0.08](0.1971154,-2.357455)(2.5971155,-2.357455)(2.5971155,-3.557455)(1.3971153,-3.557455)(1.3971153,-1.157455)(0.1971154,-1.157455)(0.1971154,-2.357455)(0.1971154,-2.357455)
\psline[linecolor=black, linewidth=0.08](10.997115,-2.357455)(10.997115,-2.357455)(10.997115,-2.357455)
\psline[linecolor=black, linewidth=0.08](6.9971156,-2.357455)(9.397116,-2.357455)(9.397116,-1.157455)(8.197115,-1.157455)(8.197115,-3.557455)(6.9971156,-3.557455)(6.9971156,-2.357455)(6.9971156,-2.357455)
\psline[linecolor=black, linewidth=0.08](13.797115,-2.357455)(13.797115,-1.157455)(14.997115,-1.157455)(14.997115,-2.357455)(14.997115,-3.557455)(16.197115,-3.557455)(16.197115,-2.357455)(13.797115,-2.357455)(13.797115,-2.357455)
\psline[linecolor=blue, linewidth=0.08](4.1971154,-1.157455)(5.397115,-1.157455)(5.397115,-1.157455)
\psline[linecolor=blue, linewidth=0.08](10.997115,-3.557455)(12.197115,-3.557455)(12.197115,-3.557455)
\psline[linecolor=red, linewidth=0.08](4.1971154,-2.357455)(5.397115,-2.357455)(5.397115,-2.357455)
\psline[linecolor=red, linewidth=0.08](10.997115,-2.357455)(12.197115,-2.357455)(12.197115,-2.357455)
\psline[linecolor=colour0, linewidth=0.08](4.1971154,-1.157455)(4.1971154,-2.357455)(4.1971154,-2.357455)
\psline[linecolor=colour0, linewidth=0.08](12.197115,-2.357455)(12.197115,-3.557455)(12.197115,-3.557455)
\psline[linecolor=yellow, linewidth=0.08](5.397115,-1.157455)(5.397115,-2.357455)(5.397115,-2.357455)
\psline[linecolor=yellow, linewidth=0.08](10.997115,-2.357455)(10.997115,-3.557455)(10.997115,-3.557455)
\psdots[linecolor=black, dotsize=0.4](4.1971154,-2.357455)
\psdots[linecolor=black, dotsize=0.4](4.1971154,-1.157455)
\psdots[linecolor=black, dotsize=0.4](5.397115,-1.157455)
\psdots[linecolor=black, dotsize=0.4](5.397115,-2.357455)
\psdots[linecolor=black, dotsize=0.4](10.997115,-2.357455)
\psdots[linecolor=black, dotsize=0.4](12.197115,-2.357455)
\psdots[linecolor=black, dotsize=0.4](12.197115,-3.557455)
\psdots[linecolor=black, dotsize=0.4](10.997115,-3.557455)
\psline[linecolor=colour1, linewidth=0.08, linestyle=dashed, dash=0.17638889cm 0.10583334cm](7.397115,-0.357455)(8.997115,-4.357455)(8.997115,-4.357455)
\rput[bl](0.5971154,-5.157455){(1)}
\rput[bl](1.7971154,-5.157455){(2)}
\rput[bl](4.5971155,-5.157455){(i)}
\rput[bl](7.397115,-5.157455){(k)}
\rput[bl](8.5971155,-5.157455){(k+1)}
\rput[bl](15.397116,-5.157455){(2k)}
\rput[bl](13.797115,-5.157455){(2k-1)}
\rput[bl](10.997115,-5.157455){(2k-i+1)}
\rput[bl](5.397115,-0.757455){$u_i$}
\rput[bl](5.397115,-3.157455){$v_i$}
\rput[bl](3.7971153,-3.157455){$w_i$}
\rput[bl](3.7971153,-0.757455){$x_i$}
\rput[bl](10.237116,-4.177455){$u_{2k-i+1}$}
\rput[bl](10.197115,-2.017455){$v_{2k-i+1}$}
\rput[bl](11.997115,-2.037455){$w_{2k-i+1}$}
\rput[bl](12.057116,-4.197455){$x_{2k-i+1}$}
\end{pspicture}
}
\end{center}
\caption{Para-chain square cactus $O_{2k}$ } \label{o2k-paraChainsqu2k}
\end{figure}

	\begin{proof}
		\begin{enumerate} 
\item[(i)] 
We consider the following cases:

\begin{itemize}
\item[\textbf{Case 1.}] Suppose that $n$ is even and  $n=2k$ for some $k\in \mathbb{N}$.
Now consider the $O_{2k}$ as shown in Figure \ref{o2k-paraChainsqu2k}. One can easily check that whatever happens to computation of Mostar index related to the edge $u_iv_i$ in the $(i)$-th square in $O_{2k}$, is the same as computation of Mostar index related to the edge $u_{2k-i+1}v_{2k-i+1}$ in the $(2k-i+1)$-th square. The same goes for $w_iv_i$ and $w_{2k-i+1}v_{2k-i+1}$, for $w_ix_i$ and $w_{2k-i+1}x_{2k-i+1}$, and also for $x_iu_i$ and $x_{2k-i+1}u_{2k-i+1}$. So for computing Mostar index, it suffices to compute the $|n_u(uv,O_{2k})-n_v(uv,O_{2k})|$  for every $uv \in E(O_{2k})$ in the first $k$ squares and then multiple that by 2. So from now, we only consider the  first $k$ squares.

Consider the yellow edge $u_iv_i$ in the $(i)$-th square. There are $3(2k)-2$ vertices which are closer to $v_i$ than $u_i$, and there is only $1$ vertex closer to $u_i$ than $v_i$ which is $x_i$. So, $|n_{u_i}(u_iv_i,O_{2k})-n_{v_i}(u_iv_i,O_{2k})|=6k-3$. By the same argument, the same happens to the edge $x_iw_i$.

Now consider the blue edge $u_ix_i$ in the $(i)$-th square. There are $3i-2$ vertices which are closer to $x_i$ than $u_i$, and there are $3k+3(k-i)+1$ vertices closer to $u_i$ than $x_i$. So, $|n_{u_i}(u_ix_i,O_{2k})-n_{x_i}(u_ix_i,O_{2k})|=6k-6i+3$. By the same argument, the same happens to the edge $v_iw_i$.

Since we have $k$ edges like blue one, $k$ edges like green one, $k$ edges like yellow one and $k$ edges like red one, then by our argument, we have:

\begin{align*}
	Mo(O_{2k})=2\left(2\sum_{i=1}^{k}3(2k-2i+1)+2\sum_{i=1}^{k}3(2k-1)\right)=36k^2-12k.
\end{align*}

\item[\textbf{Case 2.}] Suppose that $n$ is odd and $n=2k+1$ for some $k\in \mathbb{N}$.
Now consider  the $O_{2k+1}$ as shown in Figure \ref{o2k+1-paraChainsqu2k}. One can easily check that whatever happens to computation of Mostar index related to the edge $u_iv_i$ in the $(i)$-th square in $O_{2k+1}$, is the same as computation of Mostar index related to the edge $u_{2k-i+2}v_{2k-i+2}$ in the $(2k-i+2)$-th square. The same goes for $w_iv_i$ and $w_{2k-i+2}v_{2k-i+2}$, for $w_ix_i$ and $w_{2k-i+2}x_{2k-i+2}$, and also for $x_iu_i$ and $x_{2k-i+2}u_{2k-i+2}$. So for computing Mostar index, it suffices to compute the $|n_u(uv,O_{2k+1})-n_v(uv,O_{2k+1})|$  for every $uv \in E(O_{2k+1})$ in the first $k$ squares and then multiple that by 2 and add it to $\sum_{uv\in A}^{}|n_u(uv,O_{2k+1})-n_v(uv,O_{2k+1})|$, where  $A = \{ab,bc,cd,da\}$. So from now, we only consider the  first $k+1$ squares. 

Consider the yellow edge $u_iv_i$ in the $(i)$-th square. There are $3(2k+1)-2$ vertices which are closer to $v_i$ than $u_i$, and there is only $1$ vertex closer to $u_i$ than $v_i$ which is $x_i$. So, $|n_{u_i}(u_iv_i,O_{2k})-n_{v_i}(u_iv_i,O_{2k})|=6k$. By the same argument, the same happens to the edge $x_iw_i$.

Now consider the blue edge $u_ix_i$ in the $(i)$-th square. There are $3i-2$ vertices which are closer to $x_i$ than $u_i$, and there are $3(k+1)+3(k-i)+1$ vertices closer to $u_i$ than $x_i$. So, $|n_{u_i}(u_ix_i,O_{2k})-n_{x_i}(u_ix_i,O_{2k})|=6k-6i+6$. By the same argument, the same happens to the edge $v_iw_i$.

Now consider  the middle square. For the  edge $ab$, there are $3k+1$ vertices which are closer to $b$ than $a$, and there are $3k+1$ vertices closer to $a$ than $b$. the edge  $cd$ has the same attitude as $ab$. But for the  edge $ad$, there are $3(2k+1)-2$ vertices which are closer to $d$ than $a$, and there is only $1$ vertex closer to $a$ than $d$ which is $b$, and the edge  $bc$ has the same attitude as $ad$. Hence, $\sum_{uv\in A}^{}|n_u(uv,O_{2k+1})-n_v(uv,O_{2k+1})|=12k$, where  $A = \{ab,bc,cd,da\}$.

Since we have $k$ edges like blue one, $k$ edges like green one, $k$ edges like yellow one and $k$ edges like red one, then by our argument, we have:

\begin{align*}
	Mo(O_{2k+1})=2\left(2\sum_{i=1}^{k}6(k-i+1)+2\sum_{i=1}^{k}6k)\right)+12k=36k^2+24k.
\end{align*}

\end{itemize}	
Therefore, we have the result.

\item[(ii)] The proof is similar to the proof of Part (i). \qed
\end{enumerate} 	\end{proof}

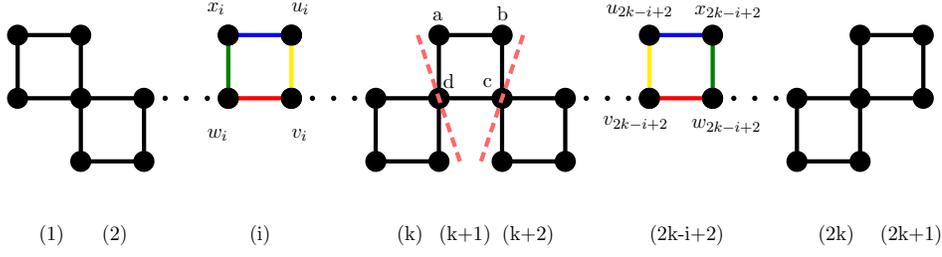
\begin{figure}
\begin{center}
\psscalebox{0.7 0.7}
{
\begin{pspicture}(0,-5.275)(17.707115,-0.625)
\definecolor{colour0}{rgb}{0.0,0.5019608,0.0}
\definecolor{colour1}{rgb}{1.0,0.4,0.4}
\psdots[linecolor=black, dotsize=0.4](0.1971154,-1.275)
\psdots[linecolor=black, dotsize=0.4](1.3971153,-1.275)
\psdots[linecolor=black, dotsize=0.4](0.1971154,-2.475)
\psdots[linecolor=black, dotsize=0.4](1.3971153,-2.475)
\psdots[linecolor=black, dotsize=0.4](2.5971155,-2.475)
\psdots[linecolor=black, dotsize=0.4](1.3971153,-3.675)
\psdots[linecolor=black, dotsize=0.4](2.5971155,-3.675)
\psdots[linecolor=black, dotsize=0.1](2.9971154,-2.475)
\psdots[linecolor=black, dotsize=0.1](3.3971155,-2.475)
\psdots[linecolor=black, dotsize=0.1](3.7971153,-2.475)
\psdots[linecolor=black, dotsize=0.4](4.1971154,-2.475)
\psdots[linecolor=black, dotsize=0.4](4.1971154,-1.275)
\psdots[linecolor=black, dotsize=0.4](5.397115,-1.275)
\psdots[linecolor=black, dotsize=0.4](5.397115,-2.475)
\psdots[linecolor=black, dotsize=0.1](5.7971153,-2.475)
\psdots[linecolor=black, dotsize=0.1](6.1971154,-2.475)
\psdots[linecolor=black, dotsize=0.1](6.5971155,-2.475)
\psdots[linecolor=black, dotsize=0.4](6.9971156,-2.475)
\psdots[linecolor=black, dotsize=0.4](8.197115,-2.475)
\psdots[linecolor=black, dotsize=0.4](6.9971156,-3.675)
\psdots[linecolor=black, dotsize=0.4](8.197115,-3.675)
\psdots[linecolor=black, dotsize=0.4](8.197115,-1.275)
\psdots[linecolor=black, dotsize=0.4](9.397116,-1.275)
\psdots[linecolor=black, dotsize=0.4](8.197115,-2.475)
\psdots[linecolor=black, dotsize=0.4](9.397116,-2.475)
\psdots[linecolor=black, dotsize=0.1](10.997115,-2.475)
\psdots[linecolor=black, dotsize=0.1](11.397116,-2.475)
\psdots[linecolor=black, dotsize=0.1](11.797115,-2.475)
\psdots[linecolor=black, dotsize=0.4](12.197115,-2.475)
\psdots[linecolor=black, dotsize=0.4](13.397116,-2.475)
\psdots[linecolor=black, dotsize=0.1](13.797115,-2.475)
\psdots[linecolor=black, dotsize=0.1](14.197115,-2.475)
\psline[linecolor=black, linewidth=0.08](0.1971154,-2.475)(2.5971155,-2.475)(2.5971155,-3.675)(1.3971153,-3.675)(1.3971153,-1.275)(0.1971154,-1.275)(0.1971154,-2.475)(0.1971154,-2.475)
\psline[linecolor=black, linewidth=0.08](10.997115,-2.475)(10.997115,-2.475)(10.997115,-2.475)
\psline[linecolor=black, linewidth=0.08](6.9971156,-2.475)(9.397116,-2.475)(9.397116,-1.275)(8.197115,-1.275)(8.197115,-3.675)(6.9971156,-3.675)(6.9971156,-2.475)(6.9971156,-2.475)
\psline[linecolor=blue, linewidth=0.08](4.1971154,-1.275)(5.397115,-1.275)(5.397115,-1.275)
\psline[linecolor=blue, linewidth=0.08](12.197115,-1.275)(13.397116,-1.275)(13.397116,-1.275)
\psline[linecolor=red, linewidth=0.08](4.1971154,-2.475)(5.397115,-2.475)(5.397115,-2.475)
\psline[linecolor=red, linewidth=0.08](12.197115,-2.475)(13.397116,-2.475)(13.397116,-2.475)
\psline[linecolor=colour0, linewidth=0.08](4.1971154,-1.275)(4.1971154,-2.475)(4.1971154,-2.475)
\psline[linecolor=colour0, linewidth=0.08](13.397116,-1.275)(13.397116,-2.475)(13.397116,-2.475)
\psline[linecolor=yellow, linewidth=0.08](5.397115,-1.275)(5.397115,-2.475)(5.397115,-2.475)
\psline[linecolor=yellow, linewidth=0.08](12.197115,-1.275)(12.197115,-2.475)(12.197115,-2.475)
\psdots[linecolor=black, dotsize=0.4](4.1971154,-2.475)
\psdots[linecolor=black, dotsize=0.4](4.1971154,-1.275)
\psdots[linecolor=black, dotsize=0.4](5.397115,-1.275)
\psdots[linecolor=black, dotsize=0.4](5.397115,-2.475)
\psdots[linecolor=black, dotsize=0.4](12.197115,-2.475)
\psdots[linecolor=black, dotsize=0.4](13.397116,-2.475)
\psdots[linecolor=black, dotsize=0.4](13.397116,-1.275)
\psdots[linecolor=black, dotsize=0.4](12.197115,-1.275)
\rput[bl](0.5971154,-5.275){(1)}
\rput[bl](1.7971154,-5.275){(2)}
\rput[bl](4.5971155,-5.275){(i)}
\rput[bl](7.397115,-5.275){(k)}
\rput[bl](8.197115,-5.275){(k+1)}
\rput[bl](5.397115,-0.875){$u_i$}
\rput[bl](5.397115,-3.275){$v_i$}
\rput[bl](3.7971153,-3.275){$w_i$}
\rput[bl](3.7971153,-0.875){$x_i$}
\rput[bl](12.197115,-5.275){(2k-i+2)}
\psdots[linecolor=black, dotsize=0.4](10.5971155,-2.475)
\psdots[linecolor=black, dotsize=0.4](9.397116,-3.675)
\psdots[linecolor=black, dotsize=0.4](10.5971155,-3.675)
\psline[linecolor=black, linewidth=0.08](9.397116,-2.475)(10.5971155,-2.475)(10.5971155,-3.675)(9.397116,-3.675)(9.397116,-2.475)(9.397116,-2.475)
\psline[linecolor=colour1, linewidth=0.08, linestyle=dashed, dash=0.17638889cm 0.10583334cm](7.7971153,-1.275)(8.5971155,-3.675)(8.5971155,-3.675)
\psline[linecolor=colour1, linewidth=0.08, linestyle=dashed, dash=0.17638889cm 0.10583334cm](9.797115,-1.275)(8.997115,-3.675)(8.997115,-3.675)
\rput[bl](8.077115,-0.995){a}
\rput[bl](9.297115,-0.995){b}
\rput[bl](9.037115,-2.255){c}
\rput[bl](8.277116,-2.275){d}
\rput[bl](9.397116,-5.275){(k+2)}
\rput[bl](11.377115,-0.955){$u_{2k-i+2}$}
\rput[bl](11.317116,-3.075){$v_{2k-i+2}$}
\rput[bl](12.997115,-3.135){$w_{2k-i+2}$}
\rput[bl](13.057116,-0.975){$x_{2k-i+2}$}
\rput[bl](16.577116,-5.275){(2k+1)}
\rput[bl](15.397116,-5.275){(2k)}
\psdots[linecolor=black, dotsize=0.4](14.997115,-2.475)
\psdots[linecolor=black, dotsize=0.4](16.197115,-2.475)
\psdots[linecolor=black, dotsize=0.4](14.997115,-3.675)
\psdots[linecolor=black, dotsize=0.4](16.197115,-3.675)
\psdots[linecolor=black, dotsize=0.4](16.197115,-1.275)
\psdots[linecolor=black, dotsize=0.4](17.397116,-1.275)
\psdots[linecolor=black, dotsize=0.4](17.397116,-2.475)
\psline[linecolor=black, linewidth=0.08](14.997115,-2.475)(17.397116,-2.475)(17.397116,-1.275)(16.197115,-1.275)(16.197115,-3.675)(14.997115,-3.675)(14.997115,-2.475)(14.997115,-2.475)
\psdots[linecolor=black, dotsize=0.1](14.5971155,-2.475)
\end{pspicture}
}
\end{center}
\caption{Para-chain square cactus $O_{2k+1}$ } \label{o2k+1-paraChainsqu2k}
\end{figure}

By the same argument as the proof of Theorem \ref{thm-para-O}, we have:

	\begin{theorem}
Let $O^h_n$ be the Ortho-chain graph  of order $n$ (See Figure \ref{ortho-chain}). Then for every $n\geq 1$, and $k \geq 1$, we have:
\begin{enumerate} 
	\item[(i)] \[
 	Mo(O^h_n)=\left\{
  	\begin{array}{ll}
  	{\displaystyle
  		100k^2-40k}&
  		\quad\mbox{if $n=2k$, }\\[15pt]
  		{\displaystyle
  			100k^2+60k}&
  			\quad\mbox{if $n=2k+1$.}
  				  \end{array}
  					\right.	
  					\]
\item[(ii)] 
\[
Mo_e(O^h_n)=\left\{
\begin{array}{ll}
{\displaystyle
	72k^2}&
\quad\mbox{if $n=2k$, }\\[15pt]
{\displaystyle
	72k^2+72k}&
\quad\mbox{if $n=2k+1$.}
\end{array}
\right.	
\]
\end{enumerate} 
\end{theorem}

\begin{figure}
\begin{center}
\psscalebox{0.5 0.5}
{
\begin{pspicture}(0,-6.8)(13.194231,-1.605769)
\psdots[linecolor=black, dotsize=0.4](2.1971154,-1.8028846)
\psdots[linecolor=black, dotsize=0.4](2.1971154,-4.2028847)
\psdots[linecolor=black, dotsize=0.4](2.5971155,-3.0028846)
\psdots[linecolor=black, dotsize=0.4](3.3971155,-3.0028846)
\psdots[linecolor=black, dotsize=0.4](3.7971156,-4.2028847)
\psdots[linecolor=black, dotsize=0.4](3.7971156,-1.8028846)
\psdots[linecolor=black, dotsize=0.4](5.3971157,-1.8028846)
\psdots[linecolor=black, dotsize=0.4](5.7971153,-3.0028846)
\psdots[linecolor=black, dotsize=0.4](5.3971157,-4.2028847)
\psdots[linecolor=black, dotsize=0.4](0.59711546,-1.8028846)
\psdots[linecolor=black, dotsize=0.4](0.19711548,-3.0028846)
\psdots[linecolor=black, dotsize=0.4](0.59711546,-4.2028847)
\psdots[linecolor=black, dotsize=0.4](1.7971154,-5.4028845)
\psdots[linecolor=black, dotsize=0.4](2.1971154,-6.6028843)
\psdots[linecolor=black, dotsize=0.4](3.7971156,-6.6028843)
\psdots[linecolor=black, dotsize=0.4](4.1971154,-5.4028845)
\psdots[linecolor=black, dotsize=0.4](6.9971156,-4.2028847)
\psdots[linecolor=black, dotsize=0.4](4.9971156,-5.4028845)
\psdots[linecolor=black, dotsize=0.4](5.3971157,-6.6028843)
\psdots[linecolor=black, dotsize=0.4](7.3971157,-5.4028845)
\psdots[linecolor=black, dotsize=0.4](6.9971156,-6.6028843)
\psdots[linecolor=black, dotsize=0.4](10.997115,-1.8028846)
\psdots[linecolor=black, dotsize=0.4](10.997115,-4.2028847)
\psdots[linecolor=black, dotsize=0.4](11.397116,-3.0028846)
\psdots[linecolor=black, dotsize=0.4](12.5971155,-4.2028847)
\psdots[linecolor=black, dotsize=0.4](9.397116,-1.8028846)
\psdots[linecolor=black, dotsize=0.4](8.997115,-3.0028846)
\psdots[linecolor=black, dotsize=0.4](9.397116,-4.2028847)
\psdots[linecolor=black, dotsize=0.4](10.5971155,-5.4028845)
\psdots[linecolor=black, dotsize=0.4](10.997115,-6.6028843)
\psdots[linecolor=black, dotsize=0.4](12.5971155,-6.6028843)
\psdots[linecolor=black, dotsize=0.4](12.997115,-5.4028845)
\psline[linecolor=black, linewidth=0.08](6.9971156,-4.2028847)(0.59711546,-4.2028847)(0.19711548,-3.0028846)(0.59711546,-1.8028846)(2.1971154,-1.8028846)(2.5971155,-3.0028846)(2.1971154,-4.2028847)(1.7971154,-5.4028845)(2.1971154,-6.6028843)(3.7971156,-6.6028843)(4.1971154,-5.4028845)(3.7971156,-4.2028847)(3.3971155,-3.0028846)(3.7971156,-1.8028846)(5.3971157,-1.8028846)(5.7971153,-3.0028846)(5.3971157,-4.2028847)(4.9971156,-5.4028845)(5.3971157,-6.6028843)(6.9971156,-6.6028843)(7.3971157,-5.4028845)(6.9971156,-4.2028847)(6.9971156,-4.2028847)
\psline[linecolor=black, linewidth=0.08](9.397116,-4.2028847)(12.5971155,-4.2028847)(12.997115,-5.4028845)(12.5971155,-6.6028843)(10.997115,-6.6028843)(10.5971155,-5.4028845)(11.397116,-3.0028846)(10.997115,-1.8028846)(9.397116,-1.8028846)(8.997115,-3.0028846)(9.397116,-4.2028847)(9.397116,-4.2028847)
\psline[linecolor=black, linewidth=0.08](7.3971157,-4.2028847)(6.9971156,-4.2028847)(6.9971156,-4.2028847)
\psline[linecolor=black, linewidth=0.08](8.997115,-4.2028847)(9.397116,-4.2028847)(9.397116,-4.2028847)
\psdots[linecolor=black, dotsize=0.1](8.197116,-4.2028847)
\psdots[linecolor=black, dotsize=0.1](7.7971153,-4.2028847)
\psdots[linecolor=black, dotsize=0.1](8.5971155,-4.2028847)
\end{pspicture}
}
\end{center}
\caption{Ortho-chain graph  $O^h_n$ } \label{ortho-chain}
\end{figure}
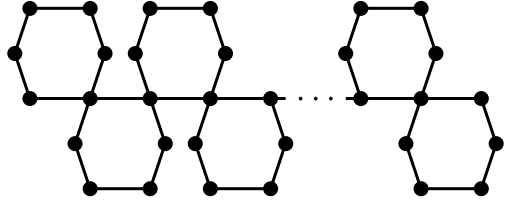

By the same argument as the proof of Theorem \ref{thm-para-Q}, we have:

	\begin{theorem}
Let $L_n$ be the para-chain hexagonal graph  of order $n$ (See Figure \ref{para-chain}). Then for every $n\geq 1$, and $k \geq 1$, we have:
 \begin{enumerate} 
 	\item[(i)] 
 	\[
 	Mo(L_n)=\left\{
  	\begin{array}{ll}
  	{\displaystyle
  		60k^2}&
  		\quad\mbox{if $n=2k$, }\\[15pt]
  		{\displaystyle
  			60k^2+60k}&
  			\quad\mbox{if $n=2k+1$.}
  				  \end{array}
  					\right.	
  					\]
\item[(ii)] 
 \[
 Mo_e(L_n)=\left\{
 \begin{array}{ll}
 {\displaystyle
 	72k^2}&
 \quad\mbox{if $n=2k$, }\\[15pt]
 {\displaystyle
 	72k^2+72k}&
 \quad\mbox{if $n=2k+1$.}
 \end{array}
 \right.	
 \]
\end{enumerate} 
	\end{theorem}

\begin{figure}
\begin{center}
\psscalebox{0.5 0.5}
{
\begin{pspicture}(0,-5.6)(16.794231,-2.805769)
\psdots[linecolor=black, dotsize=0.4](2.1971154,-3.0028846)
\psdots[linecolor=black, dotsize=0.4](2.1971154,-5.4028845)
\psdots[linecolor=black, dotsize=0.4](2.9971154,-4.2028847)
\psdots[linecolor=black, dotsize=0.4](2.9971154,-4.2028847)
\psdots[linecolor=black, dotsize=0.4](3.7971153,-5.4028845)
\psdots[linecolor=black, dotsize=0.4](3.7971153,-3.0028846)
\psdots[linecolor=black, dotsize=0.4](4.9971156,-3.0028846)
\psdots[linecolor=black, dotsize=0.4](5.7971153,-4.2028847)
\psdots[linecolor=black, dotsize=0.4](4.9971156,-5.4028845)
\psdots[linecolor=black, dotsize=0.4](0.9971154,-3.0028846)
\psdots[linecolor=black, dotsize=0.4](0.19711538,-4.2028847)
\psdots[linecolor=black, dotsize=0.4](0.9971154,-5.4028845)
\psdots[linecolor=black, dotsize=0.4](7.7971153,-3.0028846)
\psdots[linecolor=black, dotsize=0.4](7.7971153,-5.4028845)
\psdots[linecolor=black, dotsize=0.4](8.5971155,-4.2028847)
\psdots[linecolor=black, dotsize=0.4](8.5971155,-4.2028847)
\psdots[linecolor=black, dotsize=0.4](6.5971155,-3.0028846)
\psdots[linecolor=black, dotsize=0.4](5.7971153,-4.2028847)
\psdots[linecolor=black, dotsize=0.4](6.5971155,-5.4028845)
\psdots[linecolor=black, dotsize=0.4](12.997115,-3.0028846)
\psdots[linecolor=black, dotsize=0.4](12.997115,-5.4028845)
\psdots[linecolor=black, dotsize=0.4](13.797115,-4.2028847)
\psdots[linecolor=black, dotsize=0.4](13.797115,-4.2028847)
\psdots[linecolor=black, dotsize=0.4](14.5971155,-5.4028845)
\psdots[linecolor=black, dotsize=0.4](14.5971155,-3.0028846)
\psdots[linecolor=black, dotsize=0.4](15.797115,-3.0028846)
\psdots[linecolor=black, dotsize=0.4](16.597115,-4.2028847)
\psdots[linecolor=black, dotsize=0.4](15.797115,-5.4028845)
\psdots[linecolor=black, dotsize=0.4](11.797115,-3.0028846)
\psdots[linecolor=black, dotsize=0.4](10.997115,-4.2028847)
\psdots[linecolor=black, dotsize=0.4](11.797115,-5.4028845)
\psdots[linecolor=black, dotsize=0.4](8.5971155,-4.2028847)
\psdots[linecolor=black, dotsize=0.4](8.5971155,-4.2028847)
\psline[linecolor=black, linewidth=0.08](0.19711538,-4.2028847)(0.9971154,-3.0028846)(2.1971154,-3.0028846)(2.9971154,-4.2028847)(3.7971153,-3.0028846)(4.9971156,-3.0028846)(5.7971153,-4.2028847)(6.5971155,-3.0028846)(7.7971153,-3.0028846)(8.5971155,-4.2028847)(7.7971153,-5.4028845)(6.5971155,-5.4028845)(5.7971153,-4.2028847)(4.9971156,-5.4028845)(3.7971153,-5.4028845)(2.9971154,-4.2028847)(2.1971154,-5.4028845)(0.9971154,-5.4028845)(0.19711538,-4.2028847)(0.19711538,-4.2028847)
\psline[linecolor=black, linewidth=0.08](10.997115,-4.2028847)(11.797115,-3.0028846)(12.997115,-3.0028846)(13.797115,-4.2028847)(14.5971155,-3.0028846)(15.797115,-3.0028846)(16.597115,-4.2028847)(15.797115,-5.4028845)(14.5971155,-5.4028845)(13.797115,-4.2028847)(12.997115,-5.4028845)(11.797115,-5.4028845)(10.997115,-4.2028847)(10.997115,-4.2028847)
\psdots[linecolor=black, dotsize=0.1](9.397116,-4.2028847)
\psdots[linecolor=black, dotsize=0.1](9.797115,-4.2028847)
\psdots[linecolor=black, dotsize=0.1](10.197115,-4.2028847)
\end{pspicture}
}
\end{center}
\caption{Para-chain hexagonal graph  $L_n$ } \label{para-chain}
\end{figure}
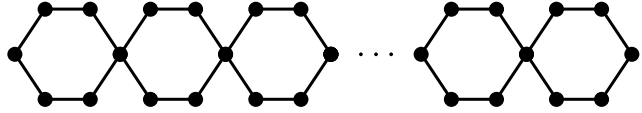

By the same argument as the proof of Theorem \ref{thm-para-O}, we have:

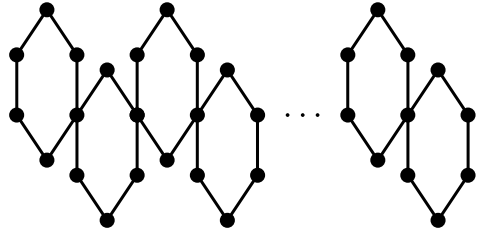
\begin{figure}
	\begin{center}
		\psscalebox{0.5 0.5}
		{
			\begin{pspicture}(0,-6.4)(12.394231,-0.40576905)
			\psdots[linecolor=black, dotsize=0.4](0.9971155,-0.60288453)
			\psdots[linecolor=black, dotsize=0.4](1.7971154,-1.8028846)
			\psdots[linecolor=black, dotsize=0.4](1.7971154,-3.4028845)
			\psdots[linecolor=black, dotsize=0.4](0.9971155,-4.6028843)
			\psdots[linecolor=black, dotsize=0.4](0.19711548,-1.8028846)
			\psdots[linecolor=black, dotsize=0.4](0.19711548,-3.4028845)
			\psdots[linecolor=black, dotsize=0.4](2.5971155,-2.2028844)
			\psdots[linecolor=black, dotsize=0.4](3.3971155,-3.4028845)
			\psdots[linecolor=black, dotsize=0.4](1.7971154,-5.0028844)
			\psdots[linecolor=black, dotsize=0.4](3.3971155,-5.0028844)
			\psdots[linecolor=black, dotsize=0.4](2.5971155,-6.2028847)
			\psdots[linecolor=black, dotsize=0.4](4.1971154,-0.60288453)
			\psdots[linecolor=black, dotsize=0.4](4.9971156,-1.8028846)
			\psdots[linecolor=black, dotsize=0.4](4.9971156,-3.4028845)
			\psdots[linecolor=black, dotsize=0.4](4.1971154,-4.6028843)
			\psdots[linecolor=black, dotsize=0.4](3.3971155,-1.8028846)
			\psdots[linecolor=black, dotsize=0.4](3.3971155,-3.4028845)
			\psdots[linecolor=black, dotsize=0.4](5.7971153,-2.2028844)
			\psdots[linecolor=black, dotsize=0.4](6.5971155,-3.4028845)
			\psdots[linecolor=black, dotsize=0.4](4.9971156,-5.0028844)
			\psdots[linecolor=black, dotsize=0.4](6.5971155,-5.0028844)
			\psdots[linecolor=black, dotsize=0.4](5.7971153,-6.2028847)
			\psdots[linecolor=black, dotsize=0.1](7.3971157,-3.4028845)
			\psdots[linecolor=black, dotsize=0.1](7.7971153,-3.4028845)
			\psdots[linecolor=black, dotsize=0.1](8.197116,-3.4028845)
			\psdots[linecolor=black, dotsize=0.4](9.797115,-0.60288453)
			\psdots[linecolor=black, dotsize=0.4](10.5971155,-1.8028846)
			\psdots[linecolor=black, dotsize=0.4](10.5971155,-3.4028845)
			\psdots[linecolor=black, dotsize=0.4](9.797115,-4.6028843)
			\psdots[linecolor=black, dotsize=0.4](8.997115,-1.8028846)
			\psdots[linecolor=black, dotsize=0.4](8.997115,-3.4028845)
			\psdots[linecolor=black, dotsize=0.4](11.397116,-2.2028844)
			\psdots[linecolor=black, dotsize=0.4](12.197116,-3.4028845)
			\psdots[linecolor=black, dotsize=0.4](10.5971155,-5.0028844)
			\psdots[linecolor=black, dotsize=0.4](12.197116,-5.0028844)
			\psdots[linecolor=black, dotsize=0.4](11.397116,-6.2028847)
			\psline[linecolor=black, linewidth=0.08](0.9971155,-0.60288453)(1.7971154,-1.8028846)(1.7971154,-5.0028844)(2.5971155,-6.2028847)(3.3971155,-5.0028844)(3.3971155,-1.8028846)(4.1971154,-0.60288453)(4.9971156,-1.8028846)(4.9971156,-5.0028844)(5.7971153,-6.2028847)(6.5971155,-5.0028844)(6.5971155,-3.4028845)(5.7971153,-2.2028844)(4.9971156,-3.4028845)(4.1971154,-4.6028843)(2.5971155,-2.2028844)(0.9971155,-4.6028843)(0.19711548,-3.4028845)(0.19711548,-1.8028846)(0.9971155,-0.60288453)(0.9971155,-0.60288453)
			\psline[linecolor=black, linewidth=0.08](11.397116,-2.2028844)(9.797115,-4.6028843)(8.997115,-3.4028845)(8.997115,-1.8028846)(9.797115,-0.60288453)(10.5971155,-1.8028846)(10.5971155,-5.0028844)(11.397116,-6.2028847)(12.197116,-5.0028844)(12.197116,-3.4028845)(11.397116,-2.2028844)(11.397116,-2.2028844)
			\end{pspicture}
		}
	\end{center}
	\caption{Meta-chain hexagonal graph  $M_n$ } \label{Meta-chain}
\end{figure}

	\begin{theorem}
Let $M_n$ be the Meta-chain hexagonal  of order $n$ (See Figure \ref{Meta-chain}). Then for every $n\geq 1$, and $k \geq 1$, we have:
\begin{enumerate} 
	\item[(i)] 
	\[
 	Mo(M_n)=\left\{
  	\begin{array}{ll}
  	{\displaystyle
  		80k^2-20k}&
  		\quad\mbox{if $n=2k$, }\\[15pt]
  		{\displaystyle
  			80k^2+60k}&
  			\quad\mbox{if $n=2k+1$.}
  				  \end{array}
  					\right.	
  					\]
\item[(ii)] 
\[
Mo_e(M_n)=\left\{
\begin{array}{ll}
{\displaystyle
	72k^2}&
\quad\mbox{if $n=2k$, }\\[15pt]
{\displaystyle
	72k^2+72k}&
\quad\mbox{if $n=2k+1$.}
\end{array}
\right.	
\]
\end{enumerate} 	\end{theorem}

We intend to derive the Mostar index and edge Mostar index of the triangulane $T_k$ defined pictorially in \cite{Khalifeh}. We define $T_k$ recursively in a manner that will be useful in our approach. First we define
recursively an auxiliary family of triangulanes $G_k$ $(k\geq 1)$. Let $G_1$ be a triangle and denote one of its vertices by $y_1$. We define $G_k$ $(k\geq 2)$ as the circuit of the graphs $G_{k-1}, G_{k-1}$,
and $K_1$ and denote by $y_k$ the vertex where $K_1$ has been placed The graphs $G_1, G_2$ and $G_3$ 
	are shown in Figure \ref{triang}.

	\begin{figure}
		\begin{center}
			\psscalebox{0.5 0.5}
			{
				\begin{pspicture}(0,-7.8216667)(16.150236,-1.6783332)
				\psline[linecolor=black, linewidth=0.08](7.6847405,-2.928333)(9.28474,-2.928333)(8.48474,-4.128333)(7.6847405,-2.928333)(7.6847405,-2.928333)
				\psline[linecolor=black, linewidth=0.08](10.884741,-2.928333)(11.28474,-1.7283331)(10.48474,-1.7283331)(10.884741,-2.928333)(10.884741,-2.928333)
				\psline[linecolor=black, linewidth=0.08](5.2847404,-4.128333)(6.884741,-6.528333)(8.48474,-4.128333)(5.2847404,-4.128333)(6.884741,-6.528333)(6.884741,-6.528333)
				\psline[linecolor=black, linewidth=0.08](0.0847406,-4.128333)(1.6847405,-6.528333)(3.2847407,-4.128333)(0.0847406,-4.128333)(1.6847405,-6.528333)(1.6847405,-6.528333)
				\psline[linecolor=black, linewidth=0.08](4.4847407,-2.928333)(6.0847406,-2.928333)(5.2847404,-4.128333)(4.4847407,-2.928333)(4.4847407,-2.928333)
				\psline[linecolor=black, linewidth=0.08](14.084741,-2.928333)(15.684741,-2.928333)(14.884741,-4.128333)(14.084741,-2.928333)(14.084741,-2.928333)
				\psline[linecolor=black, linewidth=0.08](11.684741,-4.128333)(13.28474,-6.528333)(14.884741,-4.128333)(11.684741,-4.128333)(13.28474,-6.528333)(13.28474,-6.528333)
				\psline[linecolor=black, linewidth=0.08](10.884741,-2.928333)(12.48474,-2.928333)(11.684741,-4.128333)(10.884741,-2.928333)(10.884741,-2.928333)
				\psline[linecolor=black, linewidth=0.08](12.48474,-2.928333)(12.884741,-1.7283331)(12.084741,-1.7283331)(12.48474,-2.928333)(12.48474,-2.928333)
				\psline[linecolor=black, linewidth=0.08](14.084741,-2.928333)(14.48474,-1.7283331)(13.684741,-1.7283331)(14.084741,-2.928333)(14.084741,-2.928333)
				\psline[linecolor=black, linewidth=0.08](15.684741,-2.928333)(16.08474,-1.7283331)(15.28474,-1.7283331)(15.684741,-2.928333)(15.684741,-2.928333)
				\psdots[linecolor=black, dotstyle=o, dotsize=0.4, fillcolor=white](1.6847405,-6.528333)
				\psdots[linecolor=black, dotstyle=o, dotsize=0.4, fillcolor=white](6.884741,-6.528333)
				\psdots[linecolor=black, dotstyle=o, dotsize=0.4, fillcolor=white](13.28474,-6.528333)
				\rput[bl](2.0047407,-6.7016664){$y_1$}
				\rput[bl](7.1647406,-6.688333){$y_2$}
				\rput[bl](13.631408,-6.6749997){$y_3$}
				\rput[bl](1.1647406,-7.808333){\begin{LARGE}
					$G_1$
					\end{LARGE}}
				\rput[bl](6.458074,-7.768333){\begin{LARGE}
					$G_2$
					\end{LARGE}}
				\rput[bl](12.844741,-7.8216662){\begin{LARGE}
					$G_3$
					\end{LARGE}}
				\end{pspicture}
			}
		\end{center}
		\caption{Graphs $G_1$, $G_2$ and $G_3$}\label{triang}
	\end{figure}
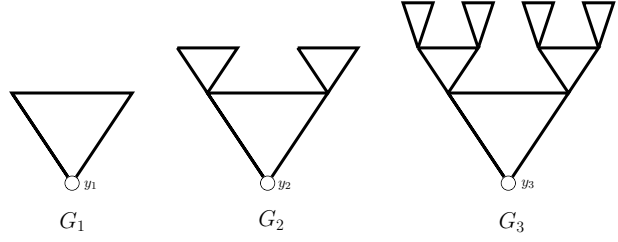
	
	\vspace{1.5cm}
	
	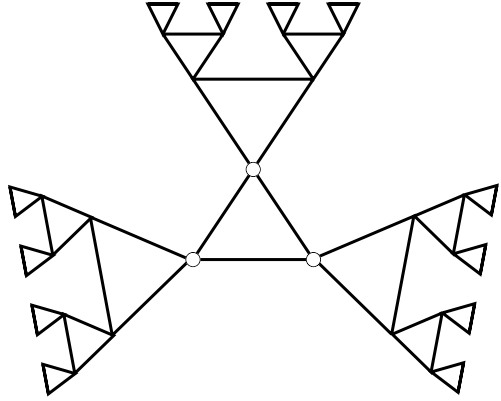
\begin{figure}
		\begin{center}
			\psscalebox{0.5 0.5}
			{
				\begin{pspicture}(0,-8.002529)(13.070843,2.5025294)
				\psline[linecolor=black, linewidth=0.08](3.7284591,2.4525292)(4.528459,2.4525292)(4.128459,1.6525292)(3.7284591,2.4525292)(4.128459,2.4525292)(4.528459,2.4525292)
				\psline[linecolor=black, linewidth=0.08](5.328459,2.4525292)(6.128459,2.4525292)(5.728459,1.6525292)(5.328459,2.4525292)(5.728459,2.4525292)(6.128459,2.4525292)
				\psline[linecolor=black, linewidth=0.08](4.128459,1.6525292)(5.728459,1.6525292)(4.928459,0.4525293)(4.128459,1.6525292)(4.128459,1.6525292)
				\psline[linecolor=black, linewidth=0.08](6.928459,2.4525292)(7.728459,2.4525292)(7.328459,1.6525292)(6.928459,2.4525292)(7.328459,2.4525292)(7.728459,2.4525292)
				\psline[linecolor=black, linewidth=0.08](8.528459,2.4525292)(9.328459,2.4525292)(8.928459,1.6525292)(8.528459,2.4525292)(8.928459,2.4525292)(9.328459,2.4525292)
				\psline[linecolor=black, linewidth=0.08](7.328459,1.6525292)(8.928459,1.6525292)(8.128459,0.4525293)(7.328459,1.6525292)(7.328459,1.6525292)
				\psline[linecolor=black, linewidth=0.08](8.128459,0.4525293)(4.928459,0.4525293)(6.528459,-1.9474707)(8.128459,0.4525293)(8.128459,0.4525293)
				\psline[linecolor=black, linewidth=0.08](13.0095825,-2.376164)(12.861903,-3.162415)(12.149491,-2.62161)(13.0095825,-2.376164)(12.935742,-2.7692895)(12.861903,-3.162415)
				\psline[linecolor=black, linewidth=0.08](12.714223,-3.948666)(12.566544,-4.734917)(11.854133,-4.194112)(12.714223,-3.948666)(12.640384,-4.3417916)(12.566544,-4.734917)
				\psline[linecolor=black, linewidth=0.08](12.144952,-2.6238508)(11.849592,-4.196353)(10.817896,-3.1885824)(12.144952,-2.6238508)(12.144952,-2.6238508)
				\psline[linecolor=black, linewidth=0.08](12.418864,-5.521168)(12.271184,-6.3074193)(11.558773,-5.766614)(12.418864,-5.521168)(12.345024,-5.914294)(12.271184,-6.3074193)
				\psline[linecolor=black, linewidth=0.08](12.123505,-7.0936704)(11.975825,-7.8799214)(11.263413,-7.339116)(12.123505,-7.0936704)(12.0496645,-7.486796)(11.975825,-7.8799214)
				\psline[linecolor=black, linewidth=0.08](11.554234,-5.768855)(11.258874,-7.341357)(10.227177,-6.3335867)(11.554234,-5.768855)(11.554234,-5.768855)
				\psline[linecolor=black, linewidth=0.08](10.210442,-6.336506)(10.80116,-3.1915019)(8.147048,-4.320965)(10.210442,-6.336506)(10.210442,-6.336506)
				\psline[linecolor=black, linewidth=0.08](1.0839291,-7.92159)(0.93781745,-7.135046)(1.7974172,-7.3822064)(1.0839291,-7.92159)(1.0108732,-7.528318)(0.93781745,-7.135046)
				\psline[linecolor=black, linewidth=0.08](0.7917058,-6.348502)(0.6455942,-5.5619583)(1.505194,-5.8091187)(0.7917058,-6.348502)(0.71865,-5.95523)(0.6455942,-5.5619583)
				\psline[linecolor=black, linewidth=0.08](1.7854072,-7.3759704)(1.4931839,-5.802882)(2.8191113,-6.370259)(1.7854072,-7.3759704)(1.7854072,-7.3759704)
				\psline[linecolor=black, linewidth=0.08](0.49948257,-4.7754145)(0.35337096,-3.9888704)(1.2129707,-4.2360306)(0.49948257,-4.7754145)(0.42642677,-4.382142)(0.35337096,-3.9888704)
				\psline[linecolor=black, linewidth=0.08](0.20725933,-3.2023263)(0.061147712,-2.4157825)(0.92074746,-2.6629426)(0.20725933,-3.2023263)(0.13420352,-2.8090544)(0.061147712,-2.4157825)
				\psline[linecolor=black, linewidth=0.08](1.2009606,-4.2297945)(0.9087374,-2.6567066)(2.234665,-3.224083)(1.2009606,-4.2297945)(1.2009606,-4.2297945)
				\psline[linecolor=black, linewidth=0.08](2.197415,-3.215118)(2.7818615,-6.361294)(4.8492703,-4.3498707)(2.197415,-3.215118)(2.197415,-3.215118)
				\psline[linecolor=black, linewidth=0.08](6.528459,-1.9474707)(4.928459,-4.3474708)(8.128459,-4.3474708)(6.528459,-1.9474707)(6.528459,-1.9474707)
				\psdots[linecolor=black, dotstyle=o, dotsize=0.4, fillcolor=white](6.528459,-1.9474707)
				\psdots[linecolor=black, dotstyle=o, dotsize=0.4, fillcolor=white](4.928459,-4.3474708)
				\psdots[linecolor=black, dotstyle=o, dotsize=0.4, fillcolor=white](8.128459,-4.3474708)
				\end{pspicture}
			}
		\end{center}
		\caption{Graph $T_3$}\label{T3}
	\end{figure}

	\begin{theorem} 
		For the graph  $T_n$ (see $T_3$ in Figure \ref{T3}), we have:
		\begin{align*}
		Mo(T_n)=6(2^{n+2}-2^n)+ \sum_{i=2}^{n}3(2^i)\left((2^{n+2}+\sum_{t=0}^{i-2}2^{n-t})-2^{n-i+1}\right).
		\end{align*}
	\end{theorem}

	\begin{figure}
		\begin{center}
			\psscalebox{0.45 0.45}
			{
				\begin{pspicture}(0,-8.2)(17.6,5.8)
				\definecolor{colour0}{rgb}{0.0,0.5019608,0.0}
				\psline[linecolor=black, linewidth=0.08](8.8,-1.4)(7.2,-3.8)(10.4,-3.8)(8.8,-1.4)(8.8,-1.4)
				\psline[linecolor=blue, linewidth=0.08](8.8,-1.4)(10.0,0.2)(10.0,0.2)
				\psline[linecolor=red, linewidth=0.08](10.0,0.2)(10.8,1.0)(10.8,1.0)
				\psline[linecolor=colour0, linewidth=0.08](10.8,1.0)(11.6,1.8)(11.6,1.8)
				\pscircle[linecolor=black, linewidth=0.08, linestyle=dotted, dotsep=0.10583334cm, dimen=outer](8.8,2.2){3.6}
				\pscircle[linecolor=black, linewidth=0.08, linestyle=dotted, dotsep=0.10583334cm, dimen=outer](14.0,-4.6){3.6}
				\pscircle[linecolor=black, linewidth=0.08, linestyle=dotted, dotsep=0.10583334cm, dimen=outer](3.6,-4.6){3.6}
				\psline[linecolor=black, linewidth=0.08](10.0,0.2)(8.0,0.2)(7.6,0.2)(8.8,-1.4)(8.8,-1.4)
				\psline[linecolor=black, linewidth=0.08](10.8,1.0)(9.2,1.0)(10.0,0.2)(10.0,0.2)
				\psline[linecolor=black, linewidth=0.08](11.6,1.8)(10.4,1.8)(10.0,1.8)(10.8,1.0)(10.8,1.0)
				\psdots[linecolor=black, dotsize=0.1](7.2,0.6)
				\psdots[linecolor=black, dotsize=0.1](6.8,1.0)
				\psdots[linecolor=black, dotsize=0.1](6.4,1.4)
				\psdots[linecolor=black, dotsize=0.1](8.8,1.4)
				\psdots[linecolor=black, dotsize=0.1](8.4,1.8)
				\psdots[linecolor=black, dotsize=0.1](8.0,2.2)
				\psdots[linecolor=black, dotsize=0.1](9.6,2.2)
				\psdots[linecolor=black, dotsize=0.1](9.2,2.6)
				\psdots[linecolor=black, dotsize=0.1](8.8,3.0)
				\psdots[linecolor=black, dotsize=0.1](11.74,2.02)
				\psdots[linecolor=black, dotsize=0.1](11.82,2.18)
				\psdots[linecolor=black, dotsize=0.1](11.92,2.36)
				\psdots[linecolor=black, dotstyle=o, dotsize=0.4, fillcolor=white](8.8,-1.4)
				\psdots[linecolor=black, dotstyle=o, dotsize=0.4, fillcolor=white](7.2,-3.8)
				\psdots[linecolor=black, dotstyle=o, dotsize=0.4, fillcolor=white](10.4,-3.8)
				\psdots[linecolor=black, dotstyle=o, dotsize=0.4, fillcolor=white](10.0,0.2)
				\psdots[linecolor=black, dotstyle=o, dotsize=0.4, fillcolor=white](10.8,1.0)
				\psdots[linecolor=black, dotstyle=o, dotsize=0.4, fillcolor=white](11.6,1.8)
				\rput[bl](8.4,3.8){$G_n$}
				\rput[bl](13.86,-4.84){$G_n$}
				\rput[bl](3.16,-4.74){$G_n$}
				\rput[bl](8.66,-2.02){$x_0$}
				\rput[bl](10.14,-0.22){$x_1$}
				\rput[bl](10.98,0.62){$x_2$}
				\rput[bl](11.7,1.28){$x_3$}
				\rput[bl](9.86,-3.7){u}
				\rput[bl](7.5,-3.66){v}
				\psdots[linecolor=black, dotstyle=o, dotsize=0.4, fillcolor=white](7.6,0.2)
				\psdots[linecolor=black, dotstyle=o, dotsize=0.4, fillcolor=white](9.2,1.0)
				\psdots[linecolor=black, dotstyle=o, dotsize=0.4, fillcolor=white](10.0,1.8)
				\rput[bl](7.1,0.06){a}
				\rput[bl](8.7,0.82){b}
				\rput[bl](9.58,1.48){c}
				\end{pspicture}
			}
		\end{center}
		\caption{Graph $T_n$}\label{Tn}
	\end{figure}
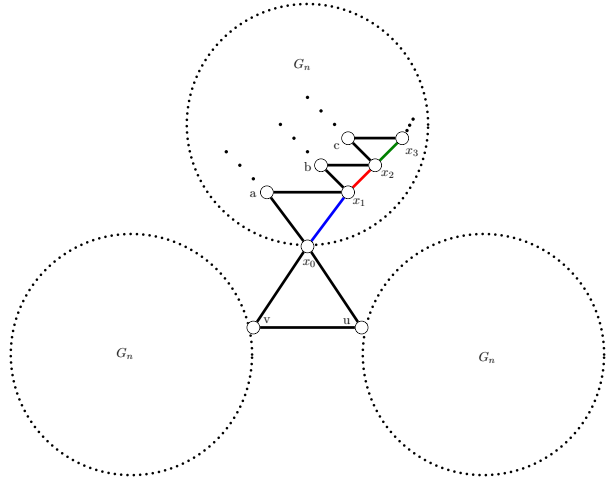

	\begin{proof}
			Consider the graph $T_n$ in Figure \ref{Tn}. First we consider the edge $x_0x_1$. There are $2(2^{n+1}-1)$ vertices which are closer to $x_o$ than $x_1$, and there are $2^n-2$ vertices closer to $x_1$ than $x_o$. So, $|n_{x_o}(x_0x_1,T_{n})-n_{x_1}(x_0x_1,T_{n})|=2^{n+2}-2^n$. The edge $ax_0$ has the same attitude as the blue edge $x_0x_1$. In total there are 6 edges with this value related to Mostar index. The number of vertices closer to vertex $a$ is the same as the number of vertices closer to vertex $x_1$, and in total, we have 3 edges like this one.
		
		Now consider  the edge $x_1x_2$. There are $2(2^{n+1}-1)+2^n$ vertices which are closer to $x_1$ than $x_2$, and there are $2^{n-1}-2$ vertices closer to $x_2$ than $x_1$. So, $|n_{x_o}(x_0x_1,T_{n})-n_{x_1}(x_0x_1,T_{n})|=2^{n+2}+2^{n+1}-2^{n-1}$. The edge $bx_1$ has the same attitude as the red edge $x_1x_2$. In total there are 12 edges with this value related to Mostar index. The number of vertices closer to vertex $b$ is the same as the number of vertices closer to vertex $x_2$, and in total, we have 6 edges like this one.
		
		By continuing this process in the $i$-th level, we have:
		$$|n_{x_{i-1}}(x_{i-1}x_i,T_{n})-n_{x_i}(x_{i-1}x_i,T_{n})|=(2^{n+2}+\sum_{t=0}^{i-2}2^{n-t})-2^{n-i+1}.$$
		We have $3(2^i)$ edges like this one. The number of vertices closer to vertex $x_i$ is the same as the number of vertices closer to its neighbour in horizontal edge with one endpoint $x_i$, and in total, we have $3(2^{i-1})$ edges like this one.
		
		Finally, the number of vertices closer to vertex $x_0$ is the same as the number of vertices closer to vertex $u$,  the number of vertices closer to vertex $x_0$ is the same as the number of vertices closer to vertex $v$, and the number of vertices closer to vertex $v$ is the same as the number of vertices closer to vertex $u$. 
		
		So by the definition of the Mostar index and our argument, we have 
		
		\begin{align*}
		Mo(T_n)=6(2^{n+2}-2^n)+ \sum_{i=2}^{n}3(2^i)\left((2^{n+2}+\sum_{t=0}^{i-2}2^{n-t})-2^{n-i+1}\right),
		\end{align*}
		and therefore we have the result. \qed
	\end{proof}

	

\end{document}